\documentclass[UTF8]{article}
\usepackage{amsmath,enumerate,amsfonts,amssymb,color,graphicx,amsthm}
\renewcommand{\raggedright}{\leftskip=0pt \rightskip=0pt plus 0cm}

\usepackage{multirow,caption,subcaption}
\usepackage{indentfirst}
\usepackage{slashed}
\usepackage{mathrsfs}
\usepackage{textcomp}
\usepackage{extarrows}
\usepackage{authblk}
\usepackage{titlesec}
\allowdisplaybreaks[4]
\usepackage{textcomp}
\usepackage{fancyhdr}
\usepackage{verbatim}
\usepackage{hyperref}
\usepackage{framed}
\usepackage[normalem]{ulem}
\usepackage{tikz}
\usetikzlibrary{patterns}
\usepackage[margin=1.25in]{geometry}
\usepackage{cite}
\hypersetup{
	colorlinks,
	citecolor=blue,
	filecolor=blue,
	linkcolor=blue,
	urlcolor=black
}

\numberwithin{equation}{section}
\newtheorem{thm}{Theorem}[section]
\newtheorem{lem}[thm]{Lemma}            
\newtheorem{prop}[thm]{Proposition}     
\newtheorem{defn}[thm]{Definition}      
\newtheorem{cor}[thm]{Corollary}        
\newtheorem{rem}[thm]{Remark}

  {\par\noindent\textbf{\emph{Proof of #1.}}\par\nobreak}
  {\par\medskip}

\newcommand{\R}{\mathbb{R}}

\newcommand{\be}{\begin{equation}}
\newcommand{\ee}{\end{equation}}

\newif\ifshowintroduction
\showintroductionfalse   

\numberwithin{equation}{section}

\begin{document}
\title{\bf Fixed-spinor radial solutions of a perturbed Euclidean spinorial Yamabe equation}
\author{\medskip  {\sc Ting Li},\ \ {\sc Zhongwei Tang}\footnote{Z. Tang is supported by National Science Foundation of China (12471099, 12271223).},  \ \
	{\sc Yuhang Zhang}    }

\date{}

\maketitle\begin{abstract}
We study the perturbed Euclidean spinorial Yamabe equation
\[
    D_{g_{\mathbb R^n}}\psi
    =
    \frac n2\bigl(1+\varepsilon K(x)\bigr)
    |\psi|^{2^*-2}\psi
    \qquad\text{on }\mathbb R^n,\quad n\ge2,
\]
where $2^*:=2n/(n-1)$, $\varepsilon>0$ is a small parameter and $K$ is a prescribed radial function.
We establish a finite-dimensional reduction within a radial
spinor space determined by a fixed unit spinor, reducing
the problem to the analysis of a function of the dilation
parameter.
Under suitable assumptions on $K$, we obtain explicit
existence criteria for solutions with this prescribed
spinorial structure, together with a multiplicity criterion
yielding solutions associated with distinct scales.
These criteria are expressed in terms of the behavior
of $K$ at the origin and weighted integral quantities.
\end{abstract}


\smallskip
\section{Introduction}
Let $(M,g,\sigma)$ be an $n$-dimensional spin manifold with $n\geq 2$, and let $\Sigma M$ denote the associated complex spinor bundle. The Dirac operator $D_g$ is a first-order differential operator acting on smooth sections of $\Sigma M$,
$
D_g:C^\infty(M,\Sigma M)\to C^\infty(M,\Sigma M).
$
See Sect.~\ref{sec:preliminaries} for more details.
In this paper, we seek solutions with a prescribed spinorial symmetry to the perturbed critical Dirac equation
\begin{equation}\label{(1.1)}
D_{g_{\mathbb R^n}}\psi
=
\frac n2\bigl(1+\varepsilon K(x)\bigr)
|\psi|^{2^*-2}\psi
\qquad\text{on }\mathbb R^n,
\end{equation} which can be viewed as a perturbation of the spinorial Yamabe equation.
Here  \(\varepsilon>0\) is sufficiently small,
\(K:\mathbb R^n\to\mathbb R\) is a prescribed radial function and \(
2^*:=\frac{2n}{n-1}\) is the critical Sobolev exponent. 
When $\varepsilon=0$, equation \eqref{(1.1)} reduces to
the Euclidean spinorial Yamabe equation
\begin{equation*}
D_{g_{\mathbb{R}^n}}\psi
=\frac{n}{2}|\psi|^{2^*-2}\psi
\qquad\text{on }\mathbb{R}^n.
\end{equation*} It
admits an explicit family of solutions
\[
\psi_{\lambda,\xi,\gamma}(x)
=
\frac{2^{\frac{n-1}{2}}\lambda^{\frac{n-1}{2}}}
     {(\lambda^2+|x-\xi|^2)^{\frac n2}}
\bigl(\lambda\gamma-(x-\xi)\cdot\gamma\bigr),
\]
where $\lambda>0$, $\xi\in\mathbb{R}^n$, and
$\gamma\in\Sigma_n$ is a constant unit spinor.
These standard bubbles arise from Killing spinors on the
round sphere through conformal transformations; see \cite{BorrelliMalchiodiWu2021}.  More generally, on a closed spin manifold $M$, the spinorial Yamabe equation is closely related to the Bär–Hijazi–Lott invariant
\begin{equation*}
	\lambda_{\min}^{+}(M, [g], \sigma) := \inf\left\{ \frac{\left( \int_{M} |D_{g}\varphi|_{g}^{\frac{2n}{n+1}} d\text{vol}_{g} \right)^{\frac{n+1}{n}}}{\int_{M} (D_{g}\varphi, \varphi)_{g} d\text{vol}_{g}}  \bigg|  \varphi\in\Gamma(\Sigma M), \int_{M} (D_{g}\varphi, \varphi)_{g} d\text{vol}_{g} > 0 \right\},
\end{equation*}
see Ammann \cite{A2003,A2009}. When this variational problem admits a minimizer,
a suitable normalization yields a spinor satisfying
\begin{equation}\label{eq:spinorial yamabe}
D_g\varphi=\frac n2|\varphi|_g^{\frac{2}{n-1}}\varphi,
\qquad
\int_M|\varphi|_g^{2^*}\,d\mathrm{vol}_g
=
\left(\frac2n\lambda_{\min}^{+}(M,[g],\sigma)\right)^n.
\end{equation} For \(n\ge3\) and \(Y(M,[g])\ge0\), Hijazi’s inequality \cite{H1986}  gives
$$
\lambda_{\min}^{+}(M,[g],\sigma)
\ge
\sqrt{\frac{n}{4(n-1)}Y(M,[g])},
$$ where $Y(M,[g])$ denotes the Yamabe invariant of the
conformal class $[g]$. 
Equality holds on the round sphere. Moreover, the spinorial analogue of Aubin’s inequality, established by Ammann et al. \cite{AGHM2008}, yields
\begin{equation}\label{Aubin inequality}
\lambda_{\min}^{+}(M,[g],\sigma)
\le
\lambda_{\min}^{+}(\mathbb{S}^n,[g_{\mathbb{S}^n}],\sigma_{\mathbb{S}^n})
=
\frac n2\operatorname{Vol}(\mathbb{S}^n,g_{\mathbb{S}^n})^{\frac{1}{n}}.
\end{equation}
When this upper bound is strict, the spinorial variational
problem admits a minimizer, yielding a nontrivial solution
of the spinorial Yamabe equation \eqref{eq:spinorial yamabe}; see \cite{A2009}.
Moreover, combining this
strict bound with Hijazi's inequality gives
\[
Y(M,[g])<Y(\mathbb{S}^n,[g_{\mathbb{S}^n}]).
\]Thus, the classical Yamabe problem is solvable in the
conformal class $[g]$. 
 Establishing the strict inequality in \eqref{Aubin inequality} is therefore an important route to existence. However, this strict inequality has so far been established only in a few special geometric settings, and a general result is still lacking; see, e.g., \cite{G2009,H2010,ADH2014,AHM2006,SX2023}.

Variational approaches to nonlinear Dirac equations face two principal
difficulties: the strong indefiniteness of the associated energy
functional and the loss of compactness at the critical Sobolev exponent.  
A natural generalization of the spinorial Yamabe equation is
\begin{equation*}
    D_g\varphi
    =H(x)|\varphi|_g^{\frac{2}{n-1}}\varphi
    \qquad \text{on } M,
\end{equation*}
where $H:M\to\mathbb{R}$ is a prescribed function.
In dimension two, this equation is closely related to the construction
of surfaces with prescribed mean curvature. More precisely, the
spinorial Weierstrass representation relates its solutions to periodic
conformal immersions, possibly with branch points, of the universal
covering $\widetilde M$ into $\mathbb{R}^3$ with mean curvature $H$.
We refer to \cite{IsobeXu2023} and the references therein for further
details.

Several approaches have been developed to study the existence of
solutions for nonconstant coefficients. Raulot \cite{Raulot}
established an existence criterion on closed spin manifolds under
a suitable strict inequality involving a weighted variational
quantity. On the round sphere, Isobe \cite{I2013} developed
a variational perturbation method to treat coefficients close to
a positive constant. This approach was further developed by Isobe
and Xu \cite{IsobeXu2023}, who obtained multiplicity results by
combining finite-dimensional reduction with topological arguments. 
In our recent work \cite{LiTangZhang2026}, we applied this perturbative
framework to obtain existence and multiplicity
results for Euclidean spinorial Yamabe equations with
lower-order and critical perturbations. 

Motivated by these works, we exploit the radial structure of the coefficient in equation  \eqref{(1.1)} to obtain explicit existence and multiplicity criteria for solutions with a prescribed spinorial symmetry.
Symmetry has long been used to recover compactness or reduce the
dimension of critical variational problems. For the scalar critical
Laplace equation, Ambrosetti, García Azorero, and Peral
\cite{AGP1999} studied radial solutions by
restricting the perturbation problem to the scale parameter of the
standard bubbles. For spinorial Yamabe-type equations, group actions
were used by Maalaoui to construct infinitely many solutions on the
round sphere \cite{Maalaoui2016}, while symmetric solutions for a
two-dimensional critical Dirac equation were obtained by Borrelli
\cite{Borrelli2022}. More recently, Maalaoui, Sire, and Xu employed
special spinorial ansatzes and symmetry reductions in their
construction of Delaunay-type singular solutions on spheres
\cite{MaalaouiSireXu2026}.

For Dirac equations, however, a radial ansatz must account for Clifford multiplication. Fix a unit spinor \(\gamma_0\in\Sigma_n\). Applying the Dirac operator to \(f(|x|)\gamma_0\) produces the component \(f'(|x|)\frac{x}{|x|}\cdot\gamma_0\). This leads us to consider spinors of the form
\[
\psi(x)=f_1(|x|)\gamma_0+
f_2(|x|)\frac{x}{|x|}\cdot\gamma_0,
\]
where \(f_1\) and \(f_2\) are real-valued functions. We call spinors of this form fixed-spinor radial spinors and denote by \(E_{\gamma_0}\) the corresponding closed real subspace of \(\mathscr{D}^{1/2}(\mathbb R^n,\Sigma\mathbb R^n)\); see Definition~\ref{defn: fixed-spinor radial space} for its precise construction.

To construct such solutions, we perform a finite-dimensional
reduction in $E_{\gamma_0}$ around the family of centered
bubbles $\psi_\lambda:=\psi_{\lambda,0,\gamma_0}$,
$\lambda>0$.
We denote this one-dimensional critical manifold by
$\mathcal M_{\gamma_0}$ and refer to Section~\ref{sec:The fixed-spinor radial setting} for its
detailed analysis.
Let $F_\varepsilon$ denote the energy functional associated
with equation~\eqref{(1.1)}, and write
\(
    F_\varepsilon^{\gamma_0}
    :=
    F_\varepsilon\big|_{E_{\gamma_0}}.
\)
Two analytic properties are needed to justify this reduction.
First, we prove that every critical point of
$F_\varepsilon^{\gamma_0}$ is a critical point of the full
functional $F_\varepsilon$.
Second, we establish that the restricted unperturbed Hessian
is a self-adjoint Fredholm operator of index zero, with
\[
\ker\bigl((F_0^{\gamma_0})''(\psi_\lambda)\bigr)
=
T_{\psi_\lambda}\mathcal M_{\gamma_0}.
\]
The proof combines angular projections with the conformal
transfer of the problem to the round sphere.
These properties allow us to apply the Lyapunov--Schmidt
reduction and reduce the construction of solutions to
a variational problem in the dilation parameter.

While the perturbation result of Isobe and Xu \cite{IsobeXu2023} relies on a Morse assumption and
an index condition on the prescribed coefficient,
we exploit radial symmetry to obtain explicit
criteria for fixed-spinor radial solutions.
Our assumptions allow $K$ to have degenerate
critical points, and the solutions are constructed
for every prescribed unit spinor $\gamma_0$.

We now state our main results.
We write $K(x)=K(r)$, where $r=|x|$, and assume that
\begin{equation}\tag{K1} \label{ass:K1}
    K\in L^\infty(\mathbb R^n)\cap C^1(\mathbb R^n), 
    \qquad
    \int_1^\infty |K(r)|r^{n-1-\alpha}\,dr<\infty,
    \quad\text{for some }0<\alpha<n.
\end{equation}
The following statement summarizes our principal existence result.
\begin{thm}\label{thm:intro-existence}
Let $n\ge2$ and assume that \eqref{ass:K1} holds.
Suppose that one of the following conditions is satisfied:
\begin{enumerate}

\item[(i)] \(K(0)=0\) and \(K\not\equiv0\);
    \item[(ii)]
    $K\in C^2(\mathbb{R}^n)$ and $
        K(0) K''(0)>0$;

    \item[(iii)]
  $
        \int_0^\infty |K'(r)|r^n\,dr<\infty$,
        and $
        K(0)\int_0^\infty K'(r)r^n\,dr>0$.
\end{enumerate}
Then, for \(|\varepsilon|\) small, equation \eqref{(1.1)} admits a
nontrivial fixed-spinor radial solution
\(\psi_\varepsilon\in E_{\gamma_0}\).\end{thm}
Further existence and multiplicity criteria are expressed
in terms of the function
\begin{equation*}
    \Gamma(\lambda)
    :=
    \int_{\mathbb R^n}
    K(x)
    \left(
        \frac{2\lambda}{\lambda^2+|x|^2}
    \right)^n
    \,dx,
    \qquad \lambda>0.
\end{equation*}

\begin{thm}\label{thm:intro-multiplicity}
Let $n\ge2$ and assume that \eqref{ass:K1} holds.
Suppose that one of the following conditions is satisfied:
\begin{enumerate}
    \item[(i)]
    $\Gamma(1)\ne0$ and $\Gamma(1)\cdot K(0)\le0$;

    \item[(ii)]
    $\Gamma(1)=0$ and $\Gamma'(1)\ne0$.
\end{enumerate}
Then, for $|\varepsilon|$ small,
equation \eqref{(1.1)} admits a nontrivial solution in
$E_{\gamma_0}$.
Moreover, if \textnormal{(ii)} holds and
$\Gamma'(1)\cdot K(0)\ge0$, then equation \eqref{(1.1)}
admits at least two distinct nontrivial solutions in
$E_{\gamma_0}$ .
\end{thm}

The paper is organized as follows. In Section~\ref{sec:preliminaries}, we recall the
spinorial Sobolev spaces, the conformal covariance of the Dirac
operator, the stereographic identification, and the abstract
perturbation theorem. In Section~\ref{sec:The fixed-spinor radial setting}, we introduce the fixed-spinor
radial space, prove that it is a natural constraint, and establish
the Fredholm and nondegeneracy properties of the
restricted unperturbed functional. In Section~\ref{sec:The reduced problem and existence results}, we analyze the
reduced function \(\Gamma\) and prove the existence and multiplicity
results.

\section{Preliminaries}\label{sec:preliminaries}

\subsection{Spin structures and the Dirac operator}
\label{subsec:spin-geometry}

In this subsection, we briefly recall some standard facts about spin
geometry and the Dirac operator. We refer to
\cite{A2003,BHMMM2015,F2000,LM1989}
for further background. Let $(M^n,g,\sigma)$ be an $n$-dimensional spin manifold with $n\ge 2$, where $g$ is a Riemannian metric and $\sigma : P_{Spin}(M)\to P_{SO}(M)$ is a fixed spin structure. Let
$
\rho:Spin(n)\to Aut(\Sigma_n)
$
be the complex spin representation, and let
$$
\Sigma M:=P_{Spin}(M)\times_\rho \Sigma_n
$$
be the associated complex spinor bundle. The bundle $\Sigma M$ carries a natural spinorial Levi--Civita connection $\nabla$, a pointwise Hermitian scalar product $\langle\cdot,\cdot\rangle_g$, and a Clifford multiplication
$
\cdot_g:TM\otimes \Sigma M\to \Sigma M$.
These structures are compatible in the sense that
\[
\langle X\cdot_g\varphi,\psi\rangle_g=-\langle\varphi,X\cdot_g\psi\rangle_g,
\]
\[
X\langle\varphi,\psi\rangle_g=\langle\nabla_X\varphi,\psi\rangle_g+\langle\varphi,\nabla_X\psi\rangle_g,
\]
and
\[
\nabla_X (Y\cdot_g\varphi)=(\nabla_XY)\cdot_g\varphi+Y\cdot_g\nabla_X\varphi
\]
for all $X,Y\in\Gamma(TM)$ and $\varphi,\psi\in\Gamma(\Sigma M)$.

Throughout the paper, all function spaces of spinors are regarded as real vector
spaces. Accordingly, we write
$
(\varphi,\psi)_g
:=
\operatorname{Re}\langle\varphi,\psi\rangle_g
$
for the real part of the Hermitian metric. All orthogonality,
self-adjointness, and Riesz representations below are understood with
respect to the corresponding real inner products.

If \(\{e_1,\ldots,e_n\}\) is a local oriented orthonormal frame of
\(TM\), the Dirac operator is defined by
$$
D_g\psi
:=
\sum_{j=1}^n
e_j\cdot_g\nabla_{e_j}\psi.
$$
\subsection{Configuration spaces and variational setting}
\label{configuration-spaces}

We first recall the fractional Sobolev spaces of spinors on a compact
spin manifold. Since \(D_g\) is self-adjoint on
\(L^2(M,\Sigma M)\) and has compact resolvent, there exists an
\(L^2\)-orthonormal basis of eigenspinors
\(\{\varphi_k\}_{k\geq1}\) satisfying
$D_g\varphi_k=\lambda_k\varphi_k$ and $|\lambda_k|\to\infty$ as $k\to\infty$.
For $s \geq 0$,  define 
$$
\left|D_g\right|^s \varphi=\sum_{k=1}^{\infty}\left|\lambda_k\right|^s \alpha_k \varphi_k
$$
for $\varphi=\sum_{k=1}^{\infty} \alpha_k \varphi_k \in L^2(M, \Sigma M)$. Denoting the domain of $\left|D_g\right|^s$ by $H^s(M, \Sigma M)$, we see that $\varphi\in H^s(M, \Sigma M)$ if and only if $\sum_{k=1}^{\infty}\left|\lambda_k\right|^{2 s}\left|\alpha_k\right|^2<\infty$. Hence $H^s(M, \Sigma M)$ coincides with the usual Sobolev space of order $s$, that is $W^{s, 2}(M, \Sigma M)$; see \cite{A2003}. We equip $H^s(M,\Sigma M)$ with the inner product
$$
\left( \varphi,\psi\right)_{s,2}
:=\bigl(|D_g|^s\varphi,|D_g|^s\psi\bigr)_{L^2}
+
(\varphi,\psi)_{L^2},
$$
where $(\cdot,\cdot)_{L^2}$ denotes the $L^2$-inner product on spinors, and write $\|\cdot\|_{s,2}$ for the induced norm. 
Thus, for
\(\varphi=\sum_{k=1}^{\infty}\alpha_k\varphi_k\), we have
$$
\|\varphi\|_{H^s}^2
=
\sum_{k=1}^{\infty}
\bigl(1+|\lambda_k|^{2s}\bigr)|\alpha_k|^2.
$$ 
For simplicity, we write $L^q$ and $H^s$ for $L^q(M,\Sigma M)$
and $H^s(M,\Sigma M)$, respectively, and denote their norms
by $\|\cdot\|_{L^q}$ and $\|\cdot\|_{H^s}$. Let $2^*=\frac{2 n}{n-1}$ be the critical Sobolev exponent for the embedding $H^{1 / 2}(M, \Sigma M) \hookrightarrow L^q(M, \Sigma M)$, $1 \leq q \leq 2^*$.
For the round sphere,
\(
\operatorname{spec}(D_{g_{\mathbb S^n}})
=
\left\{
\pm\left(\frac n2+k\right):k\in\mathbb N_0
\right\}.
\)
Consequently,
\(
0\notin\operatorname{spec}(D_{g_{\mathbb S^n}})
\)
and, since \(D_{g_{\mathbb S^n}}\) is self-adjoint,
$
\operatorname{Dom}(D_{g_{\mathbb S^n}})
=
\operatorname{Dom}(|D_{g_{\mathbb S^n}}|)
=
H^1(\mathbb S^n,\Sigma\mathbb S^n).
$

We then introduce the corresponding Euclidean configuration space. On $\mathbb{R}^n$, we consider the fractional space
$$
\mathscr{D}^{1/2}(\mathbb{R}^n,\Sigma\mathbb{R}^n)
:=
\Bigl\{
\psi:\,
f^{1/2}\psi\in L^2,\,
|D_{g_{\mathbb{R}^n}}|^{1/2}\psi\in L^2
\Bigr\}
$$
endowed with the norm
$$
\|\psi\|_{\mathscr D^{1/2}}
:=\left(
\bigl\||D_{g_{\mathbb{R}^n}}|^{1/2}\psi\bigr\|_2^2+\|f^{1/2}\psi\|_{2}^2\right)^{1/2},
$$
where $f(x)=\frac{2}{1+|x|^2}$.
Here $D_{g_{\mathbb{R}^n}}$ denotes the Euclidean Dirac operator, and $|D_{g_{\mathbb{R}^n}}|^{1/2}$ is defined through the Fourier transform by
$
\mathscr{F}\bigl(|D_{g_{\mathbb{R}^n}}|^{1/2}\psi\bigr)(\xi)
=
|\xi|^{1/2}\mathscr{F}(\psi)(\xi)$ and $\|\left|D_{g_{\mathbb{R}^n}}\right|^{1 / 2} \psi\|_2=\||\cdot|^{1 / 2} \mathscr{F}(\psi)\|_2$. 

The Euclidean and spherical configuration spaces are related by the stereographic identification described below. Let
\(
\pi_N:
\mathbb S^n\setminus\{N\}
\longrightarrow
\mathbb R^n
\)
be the stereographic projection from the north pole \(N\). Its inverse
satisfies
\(
(\pi_N^{-1})^*g_{\mathbb S^n}
=
f^2g_{\mathbb R^n}\),
\(
f(x):=\frac{2}{1+|x|^2}.
\)
Let
\(
\mathcal H_x:
\Sigma_x\mathbb R^n
\longrightarrow
\Sigma_{\pi_N^{-1}(x)}\mathbb S^n
\)
denote the fiberwise isometry induced by the corresponding spin
identification.
For
\(\psi\in C_c^\infty(\mathbb R^n,\Sigma\mathbb R^n)\), define
\begin{equation}\label{stereographic-co}
(I\psi)(q)
:=
f(\pi_N(q))^{-\frac{n-1}{2}}
\mathcal H_{\pi_N(q)}
\bigl(\psi(\pi_N(q))\bigr),
\quad q\in\mathbb S^n\setminus\{N\}.
\end{equation}
The conformal covariance of the Dirac operator gives
\begin{equation}\label{stereographic-covariance}
D_{g_{\mathbb S^n}}(I\psi)(q)
=
f(\pi_N(q))^{-\frac{n+1}{2}}
\mathcal H_{\pi_N(q)}
\bigl(
D_{g_{\mathbb R^n}}\psi(\pi_N(q))
\bigr).
\end{equation}
Equivalently,
$\label{eq:stereographic-intertwining}
D_{g_{\mathbb S^n}}I\psi
=
(f\circ\pi_N)^{-1}
I\bigl(D_{g_{\mathbb R^n}}\psi\bigr).
$
Hence, for compactly supported smooth spinors
\(\psi,\varphi\),
\begin{equation*}
\int_{\mathbb S^n}
\bigl(I\psi,D_{g_{\mathbb S^n}}I\varphi\bigr)
\,dv_{g_{\mathbb S^n}}
=
\int_{\mathbb R^n}
\bigl(\psi,D_{g_{\mathbb R^n}}\varphi\bigr)
\,dx,
\end{equation*}
and
\begin{equation*}
\int_{\mathbb S^n}|I\psi|^{2^*}\,dv_{g_{\mathbb S^n}}
=
\int_{\mathbb R^n}|\psi|^{2^*}\,dx.
\end{equation*}
Moreover,
\[
\|I\psi\|_{L^2(\mathbb S^n)}^2
=
\int_{\mathbb R^n}f(x)|\psi(x)|^2\,dx
=
\|f^{1/2}\psi\|_{L^2(\mathbb R^n)}^2.
\]With the norms introduced above, the map \(I\) extends uniquely to
an isometric isomorphism
\begin{equation}\label{eq:stereographic-isometry}
I:
\mathscr D^{1/2}(\mathbb R^n,\Sigma\mathbb R^n)
\longrightarrow
H^{1/2}(\mathbb S^n,\Sigma\mathbb S^n);
\end{equation}
see, for example, \cite{I2013}.

We now turn to the variational setting. Let \(K\in L^\infty(\mathbb R^n)\) be radial, we associate Eq.~\eqref{(1.1)} the functional $F_\varepsilon:\mathscr D^{1/2}(\R^n,\Sigma\mathbb{R}^n)\to\R$ defined by
\begin{equation}\label{(1.2)}
F_\varepsilon(\psi)=F_0(\psi)+\varepsilon G(\psi), 
\end{equation}
where $$F_0(\psi)=\frac{1}{2} \int_{\mathbb{R}^n}\left(\psi, D_{g_{\mathbb{R}^n}} \psi\right) d \operatorname{vol}_{g_{\mathbb{R}^n}}-\frac{n-1}{4} \int_{\mathbb{R}^n}|\psi|^{2^*} d \operatorname{vol}_{g_{\mathbb{R}^n}}$$
and
$$
G(\psi)=-\frac{n-1}{4} \int_{\mathbb{R}^n} K(r)|\psi|_{g_{\mathbb{R}^n} }^{2^*} d \operatorname{vol}_{g_{\mathbb{R}^n}}.$$
Since \(K\in L^\infty(\mathbb{R}^n)\) and
\(\mathscr{D}^{1/2}\hookrightarrow L^{2^*}\), the functionals
\(F_0\), \(G\), and \(F_\varepsilon\) are well defined and of class
\(C^2\) on \(\mathscr{D}^{1/2}\).
\subsection{Finite-dimensional reduction and perturbation framework}
\label{subsec:abstract-reduction}

We next recall the abstract perturbation method that will be used
below. Let \(H\) be a Hilbert space and consider a family of
functionals
\begin{equation}\label{(2.5)}
	F_{\varepsilon}(u)=F_0(u)+\varepsilon G(u),
\end{equation}
where $F_0$, $G \in C^2(H, \mathbb{R})$ and $\varepsilon>0$. On the unperturbed functional  $F_0$ we shall make
the following assumptions:
\begin{itemize}
	\item [(A1)]$F_0$ admits a $d$-dimensional $C^2$  manifold $\mathcal{M}$, with $d\geq 1$, consisting of critical points of $F_0$;
	\item [(A2)] $\forall z \in \mathcal{M}$, $d^2 F_0(z)$ is a Fredholm operator of index zero;
	\item [(A3)] $\forall z \in \mathcal{M}$, $T_z \mathcal{M}=\operatorname{Ker} d^2 F_0(z)$.
\end{itemize}
Under assumptions (A1)--(A3), one can use the implicit function theorem to find $\omega=\omega(\varepsilon, z)\in T_z \mathcal{M}^{\perp}$ for any $z\in \mathcal{M}$ such that
$$
F_\varepsilon^{\prime}(z+\omega) \in T_z \mathcal{M}.
$$

Letting $\mathcal{M}_\varepsilon=\{z+\omega(\varepsilon,z):z\in\mathcal{M}\}$, one obtains a natural constraint for $F_\varepsilon$, locally diffeomorphic to $\mathcal{M}$. For $z\in\mathcal M$, we define the reduced functional
$$
F_\varepsilon^{red}(z):=F_\varepsilon\bigl(z+\omega(\varepsilon,z)\bigr).
$$
The following result shows that critical points of the reduced functional give rise to critical points of the original functional.
\smallskip
\begin{prop}[{Ambrosetti and Malchiodi \cite[Theorem 2.12]{AM2006}}]\label{reduction}
	Assume {\rm (A1)--(A3)} are satisfied. Let $\mathcal M_c\subset\mathcal M$ be compact and $\varepsilon>0$ small. If $z_\varepsilon\in\mathcal M_c$ is a critical point of  $F_\varepsilon^{red}$, then
	$
	u_\varepsilon:=z_\varepsilon+\omega(\varepsilon,z_\varepsilon)
	$
	is a critical point of $F_\varepsilon$ on $H$.
\end{prop}
Moreover, it can be shown that the reduced functional is well approximated in the sense that
\begin{equation}\label{eq:reduced-expansion}
	F_\varepsilon^{red}(z)=F_0(z)+\varepsilon\Gamma(z)+o(\varepsilon),
	\qquad
	dF_\varepsilon^{red}(z)=\varepsilon\,d\Gamma(z)+o(\varepsilon),
\end{equation}
uniformly on compact subsets of $\mathcal M$, where
$$
\Gamma:=G|_{\mathcal M}.$$
Since $F_0$ is constant on each connected component of $\mathcal M$, the perturbative behavior of $F_\varepsilon^{red}$ is governed by the reduced functional $\Gamma$. As a consequence, one has the following classical perturbation theorem.
\begin{thm}[Ambrosetti et al. {\cite[Theorem 2.1]{AGP1999}}]\label{scalar}
	Assume \((\mathrm{A}1)\)--\((\mathrm{A}3)\) are satisfied and suppose that there exists a critical point $\bar{z} \in \mathcal{M}$ of $\Gamma$ such that one of the following conditions holds:
	\begin{itemize}
		\item [(i)] $\bar{z}$ is nondegenerate;
		\item [(ii)]$\bar{z}$ is a local proper minimum or maximum;
		\item [(iii)] $\bar{z}$ is isolated and the local topological degree of $\nabla \Gamma$ at $\bar{z}, \operatorname{deg}_{\text {loc}}\left(\nabla \Gamma, 0\right) \neq 0$.
	\end{itemize}
	Then for $|\varepsilon|$ small, the functional $F_\varepsilon$ has a critical point $u_\varepsilon$ such that $u_\varepsilon \rightarrow \bar{z}$ as $\varepsilon\rightarrow 0$.
\end{thm}
This theorem is a standard consequence of finite-dimensional
reduction. In the present paper, the extrema of
the reduced functional need not be isolated. We shall therefore use
the following compact-set result which is the fixed-spinor specialization of
\cite[Theorem~2.3]{LiTangZhang2026}.
\begin{thm}[Compact local extremal sets]
\label{thm:compact-extremal-set}
Assume that \((\mathrm{A}1)\)--\((\mathrm{A}3)\) hold, and set
\[
\Gamma:=G|_{\mathcal M}.
\]
Suppose that there exists a relatively compact open set
\(U\Subset\mathcal M\) such that either
\[
\min_{\overline U}\Gamma
<
\min_{\partial U}\Gamma\qquad
\text{ or }
\qquad
\max_{\overline U}\Gamma
>
\max_{\partial U}\Gamma.
\]
Then, for $\varepsilon>0$ sufficiently small, the functional $F_\varepsilon$ has a critical point in $H$.
\end{thm}

\section{The fixed-spinor radial setting}\label{sec:The fixed-spinor radial setting}
In this section, we establish the variational framework
for fixed-spinor radial solutions of
equation~\eqref{(1.1)} under assumption~\eqref{ass:K1}.
We first introduce the corresponding real subspace and
prove that critical points of the restricted functional
solve the full equation. We then establish the Fredholm and nondegeneracy properties
required for the abstract perturbation theorem in Section~\ref{sec:preliminaries}.
\subsection{The Euclidean and spherical spaces}
\begin{defn}\label{defn: fixed-spinor radial space}
Fix a unit spinor $\gamma_0\in\Sigma_n$.
We define
\[
    \mathcal E_{\gamma_0}(\mathbb R^n)
    :=
    \left\{
        \psi(x)
        =
        f_1(r)\gamma_0
        +
        f_2(r)\frac{x}{r}\cdot\gamma_0
        : x \in \mathbb{R}^n, 
        f_1,f_2\in C_c^\infty((0,\infty),\mathbb R)
    \right\},
\]
where $r=|x|>0$ and $\psi(0):=0$.
The fixed-spinor radial space is defined by
\[
    E_{\gamma_0}
    :=
    \overline{\mathcal E_{\gamma_0}(\mathbb R^n)}
    ^{\|\cdot\|_{1/2,2}}
    \subset
    \mathscr{D}^{1/2}(\mathbb R^n,\Sigma\mathbb R^n).
\]
\end{defn}
\begin{rem}\label{dirac formula}
	\begin{itemize}
		\item[(i)]Since \(f_1\) and \(f_2\) are compactly supported in \((0,\infty)\), the above
spinors vanish in a neighborhood of the origin and hence are smooth compactly
supported spinors on \(\mathbb R^n\).
\item[(ii)]For $\psi\in\mathcal E_{\gamma_0}(\mathbb R^n)$,
the Clifford relations give
\[
    |\psi(x)|^2=f_1(r)^2+f_2(r)^2,
\]
and
\begin{equation}\label{(1.0)}
    D_{g_{\mathbb R^n}}\psi
    =
    -\left(
        f_2'(r)+\frac{n-1}{r}f_2(r)
    \right)\gamma_0
    +
    f_1'(r)\frac{x}{r}\cdot\gamma_0
    \in\mathcal E_{\gamma_0}(\mathbb R^n).
\end{equation}
\item[(iii)] \(E_{\gamma_0}\) is a closed real subspace of
\(\mathscr{D}^{1/2}(\mathbb R^n,\Sigma\mathbb R^n)\), and is therefore
a real Hilbert space with the induced inner product $(\cdot,\cdot)_{1/2,2}$.
	\end{itemize}
\end{rem}
\begin{lem}
\label{lem:natural-constraint-fixed-gamma}
If
\(
\psi\in E_{\gamma_0}
\)
is a critical point of the restricted functional
\(
F_\varepsilon^{\gamma_0}:=F_\varepsilon|_{E_{\gamma_0}},
\)
then \(\psi\) is a weak solution of equation~\eqref{(1.1)}.

\end{lem}

\begin{proof}
Assume that \(\psi\in E_{\gamma_0}
\)
is a critical point of the restricted functional
\(
F_\varepsilon^{\gamma_0}.
\)
for 
\(
\varphi\in C_c^\infty(\mathbb R^n,\Sigma\mathbb R^n),
\) we
define
\[
\mathcal R_\psi(\varphi)
:=
\int_{\mathbb R^n}
(\psi,D_{g_{\mathbb R^n}}\varphi)\,dx
-
\frac n2
\int_{\mathbb R^n}
(1+\varepsilon K(r))|\psi|^{2^*-2}
( \psi,\varphi)\,dx .
\] We shall prove that
\(
\mathcal R_\psi(\varphi)=0,\)
\(\forall \varphi\in C_c^\infty(\mathbb R^n,\Sigma\mathbb R^n).
\)

We first consider
$\varphi\in C_c^\infty
(\mathbb R^n\setminus\{0\},\Sigma\mathbb R^n)$.
To use the critical point condition in $E_{\gamma_0}$,
we replace $\varphi$ by its angular projection onto
the two modes $\gamma_0$ and $\omega\cdot\gamma_0$. For $r>0$, define
\[
\begin{aligned}
    a_\varphi(r)
    &:=
    \frac{1}{|\mathbb S^{n-1}|}
    \int_{\mathbb S^{n-1}}
    (\varphi(r\omega),\gamma_0)\,d\omega,\\
    b_\varphi(r)
    &:=
    \frac{1}{|\mathbb S^{n-1}|}
    \int_{\mathbb S^{n-1}}
    (\varphi(r\omega),\omega\cdot\gamma_0)\,d\omega.
\end{aligned}
\]
Set
\[
    (\Pi\varphi)(x)
    :=
    a_\varphi(r)\gamma_0
    +
    b_\varphi(r)\frac{x}{r}\cdot\gamma_0.
\]
Since $\varphi$ is smooth and compactly supported away
from the origin,
$a_\varphi,b_\varphi\in C_c^\infty((0,\infty),\mathbb R)$,
and hence
$\Pi\varphi\in\mathcal E_{\gamma_0}(\mathbb R^n)$.
We claim that
\begin{equation}\label{resolve-equiv}
    \mathcal R_\psi(\varphi)
    =
    \mathcal R_\psi(\Pi\varphi).
\end{equation}
To see this, first suppose that
\(
\psi
\in  \mathcal E_{\gamma_0}(\mathbb R^n).
\)
By Remark~\ref{dirac formula}, \(D_{g_{\mathbb R^n}}\psi=
A(r)\gamma_0+B(r)\frac{x}{r}\cdot\gamma_0
\)
for some real-valued functions \(A,B\). Since \(\varphi\) and
\(\Pi\varphi\) are compactly supported away from the origin, the formal
self-adjointness of the Dirac operator gives
\[
\int_{\mathbb R^n}(\psi,D_{g_{\mathbb R^n}}\varphi)\,dx
=
\int_{\mathbb R^n}(D_{g_{\mathbb R^n}}\psi,\varphi)\,dx, \text{ and }
\int_{\mathbb R^n}(\psi,D_{g_{\mathbb R^n}}\Pi\varphi)\,dx
=
\int_{\mathbb R^n}(D_{g_{\mathbb R^n}}\psi,\Pi\varphi)\,dx.
\]
Using the definitions of \(a_\varphi\) and \(b_\varphi\), we have
\[
\int_{\mathbb{S}^{n-1}}(D_{g_{\mathbb R^n}}\psi(r\omega),\varphi(r\omega))\,d\omega
=
\int_{\mathbb{S}^{n-1}}(D_{g_{\mathbb R^n}}\psi(r\omega),\Pi\varphi(r\omega))\,d\omega .
\]Multiplying by \(r^{n-1}\) and integrating
over \(r\in(0,\infty)\), we obtain
\[
\int_{\mathbb R^n}(\psi,D_{g_{\mathbb R^n}}\varphi)\,dx
=
\int_{\mathbb R^n}(\psi,D_{g_{\mathbb R^n}}\Pi\varphi)\,dx.
\]
Since $K$ and $|\psi|$ are radial, the same argument applies to the nonlinear term:
\[
\int_{\mathbb R^n}
(1+\varepsilon K(r))|\psi|^{2^*-2}( \psi,\varphi)\,dx
=
\int_{\mathbb R^n}
(1+\varepsilon K(r))|\psi|^{2^*-2}( \psi,\Pi\varphi)\,dx.
\]
Hence
\(
\mathcal R_\psi(\varphi)=\mathcal R_\psi(\Pi\varphi)
\)
for \(\psi\in  \mathcal E_{\gamma_0}(\mathbb R^n)\).
For a general \(\psi\in E_{\gamma_0}\), take \(\psi_k\in  \mathcal E_{\gamma_0}(\mathbb R^n)\) with
\(\psi_k\to\psi\) in \(D^{1/2}\). Then \(\psi_k\to\psi\) in \(L^{2^*}\), and
\(
|\psi_k|^{2^*-2}\psi_k
\to
|\psi|^{2^*-2}\psi\)
in \(L^{\frac{2n}{n+1}}.
\)
Therefore, we have
$
\mathcal R_{\psi_k}(\varphi)\to\mathcal R_\psi(\varphi)
\text{ and }
\mathcal R_{\psi_k}(\Pi\varphi)\to\mathcal R_\psi(\Pi\varphi).
$
Passing to the limit yields \eqref{resolve-equiv}.
Since \(\psi\) is a critical point of \(F_\varepsilon^{\gamma_0}\), we have
$
(F_\varepsilon^{\gamma_0})'(\psi)[\eta]=0
\text{ for any }
 \eta\in E_{\gamma_0}.
$
In particular, taking \(\eta=\Pi\varphi\in  \mathcal E_{\gamma_0}(\mathbb R^n)\subset E_{\gamma_0}\), we get
\(
\mathcal R_\psi(\Pi\varphi)=0.
\)
Therefore
\[
\mathcal R_\psi(\varphi)=0
\qquad
\forall
\varphi\in C_c^\infty(\mathbb R^n\setminus\{0\},\Sigma\mathbb R^n).
\]

It remains to remove the restriction that the test spinor vanish near the
origin. Let
\(
\varphi\in C_c^\infty(\mathbb R^n,\Sigma\mathbb R^n)
\)
be arbitrary. Choose a smooth cut-off function \(\chi_\delta\) such that
\[
\chi_\delta(x)=0\quad\text{for }|x|\le\delta,
\qquad
\chi_\delta(x)=1\quad\text{for }|x|\ge2\delta,
\qquad
|\nabla\chi_\delta|\le \frac{C}{\delta}.
\]
Then
\(
\chi_\delta\varphi\in
C_c^\infty(\mathbb R^n\setminus\{0\},\Sigma\mathbb R^n),
\)
and thus
\(
\mathcal R_\psi(\chi_\delta\varphi)=0.
\)
Next we show that
\[
\mathcal R_\psi((1-\chi_\delta)\varphi)\to0
\qquad
\text{as }\delta\to0.
\]
For the first part, by Hölder's inequality,
\[
\left|
\int_{\mathbb R^n}
(\psi,D_{g_{\mathbb R^n}}\bigl((1-\chi_\delta)\varphi\bigr))\,dx
\right|
\le
\|\psi\|_{L^{2^*}(B_{2\delta})}
\left\|
D_{g_{\mathbb R^n}}\bigl((1-\chi_\delta)\varphi\bigr)
\right\|_{L^{\frac{2n}{n+1}}(B_{2\delta})}.
\]
Since \(\psi\in L^{2^*}\), we have
\(
\|\psi\|_{L^{2^*}(B_{2\delta})}\to0.
\)
Moreover,
\(
D\bigl((1-\chi_\delta)\varphi\bigr)
=
(1-\chi_\delta)D\varphi
-
(\nabla\chi_\delta)\cdot\varphi,
\)
and the \(L^{\frac{2n}{n+1}}\)-norm of the right-hand side is bounded uniformly in
\(\delta\),
while
\[
\|(\nabla\chi_\delta)\cdot\varphi\|_{L^{\frac{2n}{n+1}}(B_{2\delta})}
\le
C\delta^{-1}|B_{2\delta}|^{\frac{n+1}{2n}}
=
C\delta^{\frac{n-1}{2}},
\]
which tends to \(0\) for \(n\ge2\). Hence the first part
tends to \(0\).
For the nonlinear part, since
\(
(1+\varepsilon K)|\psi|^{2^*-2}\psi\in L^{\frac{2n}{n+1}},
\)
we have
\[
\left|
\int_{\mathbb R^n}
(1+\varepsilon K(r))|\psi|^{2^*-2}
(\psi,(1-\chi_\delta)\varphi)\,dx
\right|  
\le
C
\||\psi|^{2^*-1}\|_{L^{\frac{2n}{n+1}}(B_{2\delta})}
\|(1-\chi_\delta)\varphi\|_{L^{2^*}(B_{2\delta})}
\to0.
\]
Thus
\(
\mathcal R_\psi((1-\chi_\delta)\varphi)\to0
\) and
we obtain
\[
\mathcal R_\psi(\varphi)
=
\lim_{\delta\to0}
\mathcal R_\psi(\chi_\delta\varphi)
=0 
\qquad
\forall\varphi\in C_c^\infty(\mathbb R^n,\Sigma\mathbb R^n).
\]
Hence \(\psi\) is a weak solution of equation \eqref{(1.1)}.
\end{proof}

A direct computation shows that the restricted functional \(F_0^{\gamma_0}\) possesses the
following one-dimensional critical manifold
\begin{equation}\mathcal{M}_{\gamma_0 }=\left\{\left.\psi_{\lambda}(x) = \frac{2^{\frac{n-1}{2}}\lambda^{\frac{n-1}{2}}}{(\lambda^2+|x|^2)^{\frac{n}{2}}} (\lambda -x)\cdot \gamma_0 , \lambda>0\right\}\right.. \end{equation}
To study the Hessian of $F_0^{\gamma_0}$ along
$\mathcal M_{\gamma_0}$, we transfer the problem to the
round sphere via the stereographic map $I$ introduced
in Section~\ref{configuration-spaces}. Define
\[
\mathcal W_{\gamma_0}(\mathbb{S}^n)
:=I\bigl(\mathcal E_{\gamma_0}(\mathbb R^n)\bigr),
\qquad
W_{\gamma_0}:=I(E_{\gamma_0}).
\]
Here $\mathcal W_{\gamma_0}(\mathbb{S}^n)$ is the image of the smooth
class $\mathcal E_{\gamma_0}(\mathbb R^n)$.
Since $I$ is an isometric isomorphism,
\[
W_{\gamma_0}
=
\overline{\mathcal W_{\gamma_0}(\mathbb{S}^n)}
^{\,H^{1/2}(\mathbb S^n,\Sigma\mathbb S^n)}.
\]
Thus $W_{\gamma_0}$ is a closed real subspace of
$H^{1/2}(\mathbb S^n,\Sigma\mathbb S^n)$ and hence a real
Hilbert space with the induced inner product.
Moreover, the restriction
$I|_{E_{\gamma_0}}:E_{\gamma_0}\to W_{\gamma_0}$
is an isometric isomorphism.

Define the corresponding functional on the sphere by
\[
\mathcal L(\varphi):=F_0(I^{-1}\varphi),
\qquad
\varphi\in H^{1/2}(\mathbb S^n,\Sigma\mathbb S^n).
\]
By the conformal covariance of the Dirac operator,
\begin{equation}\label{eq:spherical-functional}
\mathcal L(\varphi)
=
\frac12\int_{\mathbb S^n}
(D_{g_{\mathbb S^n}}\varphi,\varphi)
\,d\mathrm{vol}_{g_{\mathbb S^n}}
-
\frac{n-1}{4}\int_{\mathbb S^n}
|\varphi|^{2^*}
\,d\mathrm{vol}_{g_{\mathbb S^n}}.
\end{equation}
The following lemma gives an explicit description of
$\mathcal W_{\gamma_0}$.
\begin{lem}\label{lem:Wtheta-characterization}
We have
\begin{equation}\label{eq:spherical-radial-form}
\mathcal W_{\gamma_0}(\mathbb{S}^n)
=
\left\{
\Psi(\theta,\omega)
=
g_1(\theta)\widetilde\gamma_0
+
g_2(\theta)\partial_\theta\cdot\widetilde\gamma_0
:
g_1,g_2\in C_c^\infty((0,\pi),\mathbb R)
\right\},
\end{equation}
where $(\theta,\omega)\in(0,\pi)\times\mathbb S^{n-1}$
are geodesic polar coordinates, $\partial_\theta$
is the radial unit vector field and $\widetilde\gamma_0:=\mathcal H(\gamma_0)$. 
The dependence of $\widetilde\gamma_0$ on
$(\theta,\omega)$ is suppressed for simplicity.
\end{lem}
\begin{proof}
For $\varphi \in \mathcal E_{\gamma_0}(\mathbb R^n)$, assume that $$
\varphi(x)
=f_{1}(r)\,\gamma_0
+
\frac{f_{2}(r)}{r}\,x\cdot \gamma_0,
\qquad f_1,f_2\in C_c^\infty((0,\infty),\mathbb R).
$$
Let $\pi_N: \mathbb{S}^n \backslash\{N\} \rightarrow \mathbb{R}^n$ be the stereographic projection, where $N \in \mathbb{S}^n$ is the north pole. Using the ambient coordinate $\mathbb{S}^n \subset \mathbb{R}^{n+1}\left(y^1, \ldots, y^{n+1}\right)$ and $\mathbb{R}^n\left(x^1, \ldots, x^n\right)$, we have
$$
\mathbb{S}^n \backslash\{N\} \ni y \mapsto \pi_N(y)=x \in \mathbb{R}^n, \quad \text { with } x^i=\frac{ y^i}{1-y^{n+1}} \quad i=1,\dots,n.
$$
In spherical coordinates $(\theta,\phi_1,...,\phi_{n-1})$, we have $|x|=\frac{\sin \theta}{1-\cos\theta}=\cot\frac{\theta}{2}.$
 Hence, 
  \begin{equation}\label{mathH}
  	\mathcal{H}(\varphi \circ \pi_N)=f_{1}(\cot\frac{\theta}{2})\mathcal{H}(\gamma_0)+\frac{f_{2}(\cot\frac{\theta}{2})}{\cot\frac{\theta}{2}}\mathcal{H}(x\cdot\gamma_0),\end{equation}
  where \(\mathcal H\) denotes the bundle isometry whose fiber map
  $ \mathcal H_x:
  \Sigma_x\mathbb R^n
  \longrightarrow
  \Sigma_{\pi_N^{-1}(x)}\mathbb S^n
  $ was introduced in Section~\ref{configuration-spaces}. 
 Set
 $\widetilde\gamma_0:=\mathcal{H}(\gamma_0).$
 Since \(|\gamma_0|=1\), we have
 $|\widetilde\gamma_0|=|\gamma_0|=1.$

To compute $\mathcal{H}(x\cdot \gamma_0)$, we introduce the vector field
$
X(x):=\sum_{i=1}^n x_i\frac{\partial}{\partial x_i}.
$
 Under the canonical identification
$T_x\mathbb R^n\simeq \mathbb R^n$, the position vector $x$ is identified with the tangent vector
$X(x)\in T_x\mathbb R^n$. Hence
$
x\cdot \gamma_0 = X\cdot \gamma_0.
$
By the compatibility of $\mathcal{H}$ with Clifford multiplication, we therefore have
$$
\mathcal{H}(x\cdot \gamma_0)=\mathcal{H}(X\cdot \gamma_0)=\widetilde X\cdot \tilde\gamma_0,
$$
where
$
\widetilde X:=f^{-1}(x)\,d\pi_N^{-1}(X)
$
is the corresponding vector field on $\mathbb{S}^n\setminus\{N\}$.
 The inverse of the stereographic projection is given by $$\pi_N^{-1}(x)
=
\left(
\frac{2x_1}{1+|x|^2},
\dots,
\frac{2x_n}{1+|x|^2},
\frac{|x|^2-1}{1+|x|^2}
\right).$$ 
 \\Therefore,
\[d\pi_N^{-1}(X)=\sum_{i=1}^{n}\frac{2x_i(1-|x|^2)}{(1+|x|^2)^2}\frac{\partial}{\partial y_i}+
\frac{4|x|^2}{(1+|x|^2)^2}\frac{\partial}{\partial y_{n+1}},\]
and $$\widetilde X=f^{-1}(x)\,d\pi_N^{-1}(X)=\sum_{i=1}^{n}\frac{x_i(1-|x|^2)}{1+|x|^2}\frac{\partial}{\partial y_i}+
\frac{2|x|^2}{1+|x|^2}\frac{\partial}{\partial y_{n+1}}.$$
 Write
$
y=(\sin\theta\,\omega,\cos\theta)
$
with
$
\omega
=
(\cos\phi_1,\,
\sin\phi_1\cos\phi_2,\,
\dots,\,
\prod_{j=1}^{n-2}\sin\phi_j\cos\phi_{n-1},\,
\prod_{j=1}^{n-1}\sin\phi_j).
$
Then
$$
x_i=\frac{y_i}{1-y_{n+1}}=\frac{\sin\theta\,\omega_i}{1-\cos\theta}=\cot\frac{\theta}{2}\,\omega_i.
$$
Thus, we obtain
\begin{equation}\label{mathHx}
\widetilde X
=
-\sum_{i=1}^{n}\cos\theta\cot\frac{\theta}{2}\omega_i\frac{\partial}{\partial y_i}
+
2\cos^2\frac{\theta}{2}\frac{\partial}{\partial y_{n+1}}
=
-\cot\frac{\theta}{2}\frac{\partial}{\partial\theta}.
\end{equation}
and
$
\mathcal{H}(x\cdot \gamma_0)
=
\widetilde X\cdot\tilde \gamma_0
=
-\cot\frac{\theta}{2}\,\partial_\theta\cdot \tilde \gamma_0.
$
Consequently, by \eqref{stereographic-co}, \eqref{mathH} and \eqref{mathHx}, we have
$$
I(\varphi)
=
f^{-\frac{n-1}{2}}\mathcal{H}(\varphi\circ\pi_N)
=
\left(2\sin^2\frac{\theta}{2}\right)^{-\frac{n-1}{2}}
\left(
f_{1}\!\left(\cot\frac{\theta}{2}\right)\tilde\gamma_0
-f_{2}\!\left(\cot\frac{\theta}{2}\right)\partial_\theta\cdot \tilde\gamma_0
\right).$$
 Setting
$
g_{1}(\theta)
=
\left(2\sin^2\frac{\theta}{2}\right)^{-\frac{n-1}{2}}
f_{1}\!\left(\cot\frac{\theta}{2}\right)
\text{ and }
g_{2}(\theta)
=
-
\left(2\sin^2\frac{\theta}{2}\right)^{-\frac{n-1}{2}}
f_{2}\!\left(\cot\frac{\theta}{2}\right),
$
we conclude that
$$
I(\varphi)
=
g_{1}(\theta)\tilde\gamma_0
+g_{2}(\theta)\partial_\theta\cdot \tilde\gamma_0.
$$
Since the change of variable $r=\cot\frac{\theta}{2}$ is a smooth diffeomorphism from \((0,\pi)\) onto
\((0,\infty)\), and the conformal factor is smooth and positive, the above
correspondence between \((f_1,f_2)\) and \((g_1,g_2)\) is bijective.
Therefore, the conclusion follows.

\end{proof}
\subsection{The fixed-spinor projection}
We now introduce the fixed-spinor projection $P$ and establish
the properties needed to analyze the restricted Hessian. We continue to use the geodesic polar coordinates
$(\theta,\omega)\in(0,\pi)\times\mathbb S^{n-1}$
and write $d\omega$ for the standard volume element
on $\mathbb S^{n-1}$.

Set
\[
E_1(\theta,\omega):=\widetilde\gamma_0(\theta,\omega),
\qquad
E_2(\theta,\omega)
:=\partial_\theta\cdot\widetilde\gamma_0(\theta,\omega).
\]
These are the two spinor fields appearing in
Lemma~\ref{lem:Wtheta-characterization}.
The Clifford relations give
\[
|E_1|=|E_2|=1,
\qquad
(E_1,E_2)=0.
\]
For \(\Phi\in L^2(\mathbb{S}^n,\Sigma \mathbb{S}^n)\), define for a.e
\(\theta\in(0,\pi)\),
\[
a_\Phi(\theta)
:=
\frac1{|\mathbb{S}^{n-1}|}
\int_{ \mathbb{S}^{n-1}}
(
\Phi(\theta,\omega),E_1(\theta,\omega))
\,d\omega,\qquad
b_\Phi(\theta)
:=
\frac1{| \mathbb{S}^{n-1}|}
\int_{\mathbb{S}^{n-1}}(
\Phi(\theta,\omega),E_2(\theta,\omega))
\,d\omega,
\]
and
\[
P\Phi(\theta,\omega)
:=
a_\Phi(\theta)E_1(\theta,\omega)
+
b_\Phi(\theta)E_2(\theta,\omega).
\]For $s\in\{0,\tfrac12,1\}$, set
\[
W_{\gamma_0}^s
:=
\overline{\mathcal W_{\gamma_0}(\mathbb S^n)}
^{\,H^s(\mathbb S^n,\Sigma\mathbb S^n)},
\]
where $H^0(\mathbb S^n,\Sigma\mathbb S^n)=L^2(\mathbb S^n,\Sigma\mathbb S^n)$. In particular,
$W_{\gamma_0}^{1/2}=W_{\gamma_0}$.

\begin{lem}\label{lem:projection-fixed-gamma}
The projection $P$ has the following properties:
\begin{enumerate}
\item[(i)]
$P$ is the $L^2$-orthogonal projection onto
$W_{\gamma_0}^0$.

\item[(ii)]
For $s=0, \tfrac12,1$,
\(
P\bigl(H^s(\mathbb S^n,\Sigma\mathbb S^n)\bigr)
=W_{\gamma_0}^s,
\)
and there exists $C>0$ such that
\[
\|P\Phi\|_{H^s}
\le C\|\Phi\|_{H^s},\qquad \forall\,
\Phi\in H^s(\mathbb S^n,\Sigma\mathbb S^n).
\]
\end{enumerate}
\end{lem}
\begin{proof}
\noindent\textbf{Step 1. The $L^2$-orthogonal projection property.}

	We first prove that \(P\) is the \(L^2\)-orthogonal projection onto \( W_{\gamma_0}^0\). Since \(E_1\) and \(E_2\) are pointwise
orthonormal, we have
	\[\|P\Phi\|^2_{L^2}=\int_{0}^\pi\int_{ \mathbb{S}^{n-1}}|P\Phi(\theta,\omega)|^2\sin^{n-1}\theta\, d\omega d\theta=|\mathbb{S}^{n-1}|\int_0^\pi (|a_\Phi(\theta)|^2+|b_\Phi(\theta)|^2) \sin^{n-1}\theta\,d\theta.\]
  By the Cauchy--Schwarz inequality and the pointwise
orthonormality of $E_1$ and $E_2$, for a.e. $\theta$,
\[
\begin{aligned}
|a_\Phi(\theta)|^2+|b_\Phi(\theta)|^2
&\le
\frac{1}{|\mathbb S^{n-1}|}
\int_{\mathbb S^{n-1}}
\left[
(\Phi(\theta,\omega),E_1(\theta,\omega))^2
+
(\Phi(\theta,\omega),E_2(\theta,\omega))^2
\right]\,d\omega
\\
&\le
\frac{1}{|\mathbb S^{n-1}|}
\int_{\mathbb S^{n-1}}
|\Phi(\theta,\omega)|^2\,d\omega.
\end{aligned}
\]
Integrating with respect to
$\sin^{n-1}\theta\,d\theta$, we obtain
\[
\|P\Phi\|_{L^2}
\le \|\Phi\|_{L^2}.
\]
  In particular, \(a_\Phi,b_\Phi\in L^2\bigl((0,\pi),\sin^{n-1}\theta\,d\theta\bigr). \) Choose \(a_k,b_k\in C_c^\infty((0,\pi),
  \mathbb R)\) such that \( a_k\to a_\Phi\), \(b_k\to b_\Phi \) in \(L^2((0,\pi),\sin^{n-1}\theta\,d\theta)\). Then \( a_k(\theta)E_1+b_k(\theta)E_2\in \mathcal W_{\gamma_0}(\mathbb{S}^n) \) and converges to \(P\Phi\) in \(L^2(\mathbb{S}^n,\Sigma \mathbb{S}^n)\). It follows that  $$
P\Phi\in
\overline{\mathcal W_{\gamma_0}(\mathbb{S}^n)}^{\,L^2}
=
W_{\gamma_0}^0.
$$
  It remains to prove the orthogonality. Let \(\Psi\in \mathcal W_{\gamma_0}(\mathbb{S}^n)\) and write \[ \Psi(\theta,\omega)=a(\theta)E_1(\theta,\omega)+b(\theta)E_2(\theta,\omega), \qquad a,b\in C_c^\infty((0,\pi), \mathbb R). \] By the definitions of $a_\Phi$ and $b_\Phi$,
\[
\int_{\mathbb S^{n-1}}
\bigl(
\Phi(\theta,\omega)-P\Phi(\theta,\omega),
E_j(\theta,\omega)
\bigr)\,d\omega
=0,
\qquad j=1,2,
\]
for a.e. $\theta$.
Multiplying these identities by $a(\theta)$ and
$b(\theta)$, respectively, and integrating with respect
to $\sin^{n-1}\theta\,d\theta$, we obtain
\[
(\Phi-P\Phi,\Psi)_{L^2}=0.
\]Since \(\mathcal W_{\gamma_0}(\mathbb{S}^n)\) is dense in \( W_{\gamma_0}^0\), the same identity holds for all \(\Psi\in W_{\gamma_0}^0\). Thus
\(P\Phi\in W_{\gamma_0}^0\) and
\(\Phi-P\Phi\in(W_{\gamma_0}^0)^\perp\), which proves that \(P\) is the
\(L^2\)-orthogonal projection onto \(W_{\gamma_0}^0\). In
particular,
$
\|P\Phi\|_{L^2}
\leq
\|\Phi\|_{L^2}.$

\noindent\textbf{Step 2. Boundedness on $H^1$ and $H^{1/2}$.}

Let $N,S\in \mathbb{S}^n$ be the two poles and set
$
  \mathscr C
  :=C_c^\infty\bigl(\mathbb{S}^n\setminus\{N,S\},\Sigma \mathbb{S}^n\bigr).
$
We first claim that
\begin{equation}\label{N,S}
  \overline{\mathscr C}^{\,H^1(\mathbb{S}^n,\Sigma \mathbb{S}^n)}
  =
  H^1(\mathbb{S}^n,\Sigma \mathbb{S}^n).
\end{equation}
Since smooth spinors are dense in $H^1$, it suffices
to approximate an arbitrary
$\Phi\in C^\infty(\mathbb S^n,\Sigma\mathbb S^n)$
by elements of $\mathscr C$.
We construct smooth cut-off functions $\chi_\rho$
vanishing near both poles such that
\[
\|\chi_\rho\Phi-\Phi\|_{H^1}\longrightarrow0
\qquad\text{as }\rho\to0.
\]
Suppose first that \(n\geq3\). 
For \(p\in\{N,S\}\) and sufficiently small \(\rho>0\), consider \(\chi_{\rho,p}\in C^\infty(\mathbb{S}^n,[0,1])\) such that
\[
  \chi_{\rho,p}=0
  \quad\text{on }B_\rho(p),
  \qquad
  \chi_{\rho,p}=1
  \quad\text{on }\mathbb{S}^n\setminus B_{2\rho}(p),\qquad|\nabla\chi_{\rho,p}|
  \leq
  \frac{C}{\rho},
\]
and therefore
\[
  \int_{\mathbb{S}^n}|\nabla\chi_{\rho,p}|^2\,d\mathrm{vol}_{g_{\mathbb S^n}}
\le
  \frac{C}{\rho^2}
  \int_\rho^{2\rho}(\sin r)^{n-1}\,dr\leq
  C\rho^{n-2}
  \longrightarrow0.
\]
Suppose now that \(n=2\). In this case we use a logarithmic cut-off. For sufficiently small $\rho>0$, choose
$\chi_{\rho,p}\in C^\infty(\mathbb S^2,[0,1])$ such that
\[
  \chi_{\rho,p}=0
  \quad\text{on }B_{\rho^2}(p),
  \qquad
  \chi_{\rho,p}=1
  \quad\text{on }S^2\setminus B_\rho(p).
\]
On the transition annulus
\(
  \rho^2<r<\rho,
\)
one has
$
  |\nabla\chi_{\rho,p}|
  \leq
  \frac{C}{r\log(1/\rho)}.
$
Consequently,
\[
  \int_{S^2}|\nabla\chi_{\rho,p}|^2\,d\mathrm{vol}_{g_{\mathbb S^2}}
\le
  \frac{C}{\log^2(1/\rho)}
  \int_{\rho^2}^{\rho}
  \frac{\sin r}{r^2}\,dr\leq
  \frac{C}{\log(1/\rho)}
  \longrightarrow0.
\]
Applying the preceding construction at the two poles, define
\(
\chi_\rho
:=
\chi_{\rho,N}\chi_{\rho,S}.
\)
Then \(\chi_\rho\in C_c^\infty\bigl(\mathbb{S}^n\setminus\{N,S\},R\bigr)\) and satisfies
\[
  \chi_\rho=1
  \quad\text{on}\quad
  \mathbb{S}^n\setminus
  \bigl(B_{2\rho}(N)\cup B_{2\rho}(S)\bigr),
\]
with
\[
  \|\nabla\chi_\rho\|_{L^2}^2
  \leq
  \begin{cases}
    C\rho^{n-2}, & n\geq3,\\[1mm]
    \displaystyle\frac{C}{\log(1/\rho)}, & n=2.
  \end{cases}
\]
Since $\Phi$ is smooth and $\chi_\rho=1$ outside
$B_{2\rho}(N)\cup B_{2\rho}(S)$, the product rule
and the preceding estimates give
\[
\|\chi_\rho\Phi-\Phi\|_{H^1}\longrightarrow0.
\]
As $\chi_\rho\Phi\in\mathscr C$, this proves \eqref{N,S}.

Since smooth spinors are dense in $H^s$ and the embedding
$H^1\hookrightarrow H^s$ is continuous for $0\le s\le1$,
the same approximation yields
$$\overline{\mathscr C}^{\,H^s(\mathbb{S}^n,\Sigma \mathbb{S}^n)}
  =
  H^s(\mathbb{S}^n,\Sigma \mathbb{S}^n),\quad
  0\leq s\leq1.$$

 We next prove the following commutation identity:
\begin{equation}
  D_{g_{\mathbb{S}^n}}P\Phi=PD_{g_{\mathbb{S}^n}}\Phi,
  \qquad \Phi\in\mathscr C.
  \label{eq:core-commutation}
\end{equation}
If $\Phi\in\mathscr C$, then the functions $a_\Phi$ and $b_\Phi$ are
smooth and compactly supported in $(0,\pi)$. Hence,
\(
  P\Phi\in\mathcal W_{\gamma_0}(\mathbb{S}^n).
\)
Since $D_{g_{\mathbb{S}^n}}$ is a local first-order differential operator,
\(
  D_{g_{\mathbb{S}^n}}\Phi\in\mathscr C,
\)
and therefore $PD_{g_{\mathbb{S}^n}}\Phi\in\mathcal W_{\gamma_0}(\mathbb{S}^n)$.
By \eqref{eq:stereographic-intertwining}, the invariance of
\(\mathcal E_{\gamma_0}(\mathbb R^n)\) under
\(D_{g_{\mathbb R^n}}\) and the fact that the conformal factor is radial imply
\begin{equation}\label{eq:D-invariance-core}
D_{g_{\mathbb S^n}}
\bigl(\mathcal W_{\gamma_0}(\mathbb S^n)\bigr)
\subset
\mathcal W_{\gamma_0}(\mathbb S^n).
\end{equation}
It follows that both $D_{g_{\mathbb{S}^n}}P\Phi$ and $PD_{g_{\mathbb{S}^n}}\Phi$ belong to $\mathcal W_{\gamma_0}(\mathbb{S}^n)$. For
every $\varphi\in\mathcal W_{\gamma_0}(\mathbb{S}^n)$, the formal self-adjointness of $D$, the
$L^2$-projection property of $P$, and
\eqref{eq:D-invariance-core} give
\begin{align*}
  (D_{g_{\mathbb{S}^n}}P\Phi-PD_{g_{\mathbb{S}^n}}\Phi,\varphi)_{L^2}
  &=(P\Phi,D_{g_{\mathbb{S}^n}}\varphi)_{L^2}-(D_{g_{\mathbb{S}^n}}\Phi,\varphi)_{L^2}\\
  &=(\Phi,D_{g_{\mathbb{S}^n}}\varphi)_{L^2}-(D_{g_{\mathbb{S}^n}}\Phi,\varphi)_{L^2}\\
  &=0.
\end{align*}
Since $D_{g_{\mathbb{S}^n}}P\Phi-PD_{g_{\mathbb{S}^n}}\Phi\in\mathcal W_{\gamma_0}(\mathbb{S}^n)$, we may take
$\varphi=D_{g_{\mathbb{S}^n}}P\Phi-PD_{g_{\mathbb{S}^n}}\Phi$ and obtain \eqref{eq:core-commutation}.

For \(\Phi\in\mathscr C\), the spectral definition of the
\(H^1\)-norm, together with \eqref{eq:core-commutation}, gives
\[
\begin{aligned}
\|P\Phi\|_{H^1}^2
&=
\|P\Phi\|_{L^2}^2
+
\|D_{g_{\mathbb S^n}}P\Phi\|_{L^2}^2\\
&=
\|P\Phi\|_{L^2}^2
+
\|P D_{g_{\mathbb S^n}}\Phi\|_{L^2}^2\\
&\leq
\|\Phi\|_{L^2}^2
+
\|D_{g_{\mathbb S^n}}\Phi\|_{L^2}^2=
\|\Phi\|_{H^1}^2.
\end{aligned}
\]
Let now $\Phi\in H^1(\mathbb{S}^n,\Sigma \mathbb{S}^n)$ and choose
$\Phi_k\in\mathscr C$ such that $\Phi_k\to\Phi$ in $H^1$. The preceding
estimate shows that $P\Phi_k$ converges in $H^1$ to some spinor
$\widetilde\Phi$. On the other hand, since $P$ is bounded on $L^2$,
\[
  P\Phi_k\longrightarrow P\Phi
  \quad\text{in }L^2.
\]
Hence $\widetilde\Phi=P\Phi$, and therefore
\[
  \|P\Phi\|_{H^1}\leq C\|\Phi\|_{H^1}.
\]
Thus the original $L^2$-projection $P$ restricts to a bounded
projection on $H^1$.

Since $P$ is real-linear and bounded on both $L^2$ and $H^1$,
real interpolation gives its boundedness on $H^{1/2}$, using
\[
\bigl(L^2(\mathbb S^n,\Sigma\mathbb S^n),
H^1(\mathbb S^n,\Sigma\mathbb S^n)\bigr)_{1/2,2}
=H^{1/2}(\mathbb S^n,\Sigma\mathbb S^n)
\]
with equivalent norms.

\noindent\textbf{Step 3. The range of $P$.}

We now prove
\[
P(H^s(\mathbb{S}^n,\Sigma \mathbb{S}^n))= W_{\gamma_0}^s,
\qquad
s=0,\frac12,1.
\]
Let $\Psi\in W_{\gamma_0}^s$. By definition, there exist
$\Psi_k\in\mathcal W_{\gamma_0}(\mathbb S^n)$ such that
$\Psi_k\to\Psi$ in $H^s$.
Since $P\Psi_k=\Psi_k$ and $P$ is bounded on $H^s$,
passing to the limit gives $P\Psi=\Psi$.
Thus
\[
W_{\gamma_0}^s\subset P(H^s).
\]
Conversely, let
$\Phi\in H^s(\mathbb S^n,\Sigma\mathbb S^n)$.
By the density result above, choose $\Phi_k\in\mathscr C$
such that $\Phi_k\to\Phi$ in $H^s$.
Then $P\Phi_k\in\mathcal W_{\gamma_0}(\mathbb S^n)$,
and the boundedness of $P$ gives
$P\Phi_k\to P\Phi$ in $H^s$.
Hence $P\Phi\in W_{\gamma_0}^s$, so
\[
P(H^s)\subset W_{\gamma_0}^s.
\]
This completes the proof.
\end{proof}
\begin{lem}
\label{lem:commute}
For every $\Phi\in H^1(\mathbb S^n,\Sigma\mathbb S^n)$,
\(
D_{g_{\mathbb S^n}}P\Phi
=
PD_{g_{\mathbb S^n}}\Phi.
\)
\end{lem}
\begin{proof}
By Lemma~\ref{lem:projection-fixed-gamma}, $P$ is bounded on both $L^2$ and $H^1$.
Let $\Phi\in H^1(\mathbb{S}^n,\Sigma \mathbb{S}^n)$. By the density of \(\mathscr C\)
in \(H^1(\mathbb{S}^n,\Sigma \mathbb{S}^n)\), choose
$\Phi_k\in\mathscr C$ such that
\(
  \Phi_k\longrightarrow\Phi\) in \(H^1.
\)
The $H^1$-boundedness of $P$ gives
\(
  P\Phi_k\longrightarrow P\Phi
 \) in \(H^1,
\)
and hence
\(
  D_{g_{\mathbb S^n}}P\Phi_k\longrightarrow D_{g_{\mathbb S^n}}P\Phi
  \) in \(L^2.
\)
On the other hand,
\(
  D_{g_{\mathbb S^n}}\Phi_k\longrightarrow D_{g_{\mathbb S^n}}\Phi
  \) in \(L^2,
\)
so the $L^2$-boundedness of $P$ yields
\(
  PD_{g_{\mathbb S^n}}\Phi_k\longrightarrow PD_{g_{\mathbb S^n}}\Phi
  \) in \(L^2.
\)
By the commutation identity \eqref{eq:core-commutation},
\(
  D_{g_{\mathbb S^n}}P\Phi_k=PD_{g_{\mathbb S^n}}\Phi_k.
\)
Passing to the limit gives
\(
  D_{g_{\mathbb S^n}}P\Phi=PD_{g_{\mathbb S^n}}\Phi.
\)
\end{proof}
\begin{cor}\label{cor:commute-borel}
  $P$ commutes with $f(D_{g_{\mathbb S^n}})$ for every bounded real-valued Borel function $f$. In particular, with $A:=(|D_{g_{\mathbb S^n}}|+1)^{-1}D_{g_{\mathbb S^n}}$,
\[
  P(|D_{g_{\mathbb S^n}}|+1)^{-1}=(|D_{g_{\mathbb S^n}}|+1)^{-1}P,\qquad PA=AP,
\]
and consequently $W_{\gamma_0}$ is preserved by both operators:
\[
  (|D_{g_{\mathbb S^n}}|+1)^{-1}(W_{\gamma_0})\subset W_{\gamma_0},\qquad A(W_{\gamma_0})\subset W_{\gamma_0}.
\]
\end{cor}\begin{proof}
  By Lemma~\ref{lem:commute}, if $D_{g_{\mathbb S^n}}\Phi=\lambda\Phi$, then
\(
D_{g_{\mathbb S^n}}P\Phi
=PD_{g_{\mathbb S^n}}\Phi
=\lambda P\Phi.
\)
Thus $P$ preserves each eigenspace of $D_{g_{\mathbb S^n}}$.
By the spectral expansion and the boundedness of $P$ on $L^2$,
it follows that $P$ commutes with $f(D_{g_{\mathbb S^n}})$
for every bounded real-valued Borel function $f$.
Taking
\(
f(\lambda)=\frac{1}{|\lambda|+1}\)
and \(
f(\lambda)=\frac{\lambda}{|\lambda|+1},
\)
respectively, we obtain
\[
P(|D_{g_{\mathbb S^n}}|+1)^{-1}=(|D_{g_{\mathbb S^n}}|+1)^{-1}P,
\qquad
PA=AP.
\]
We next verify that both operators are bounded on \(H^{1/2}\).
Let
\(
\Phi=\sum_{k=1}^\infty \alpha_k\varphi_k\in H^{1/2}\),
 \(D_{g_{\mathbb S^n}}\varphi_k=\lambda_k\varphi_k.
\)
By the spectral characterization of the \(H^{1/2}\)-norm,
\[
\|\Phi\|_{H^{1/2}}^2
=
\sum_{k=1}^\infty
(1+|\lambda_k|)|\alpha_k|^2.
\]
Since
\[
(|D_{g_{\mathbb S^n}}|+1)^{-1}\Phi
=
\sum_{k=1}^\infty
\frac{\alpha_k}{1+|\lambda_k|}\varphi_k,
\]
it follows that
\[
\bigl\|(|D_{g_{\mathbb S^n}}|+1)^{-1}\Phi\bigr\|_{H^{1/2}}^2
=
\sum_{k=1}^\infty
(1+|\lambda_k|)
\left|
\frac{\alpha_k}{1+|\lambda_k|}
\right|^2\leq
\sum_{k=1}^\infty
(1+|\lambda_k|)|\alpha_k|^2=
\|\Phi\|_{H^{1/2}}^2.
\]Similarly, since
\[
A\Phi
=
\sum_{k=1}^{\infty}
\frac{\lambda_k}{1+|\lambda_k|}
\alpha_k\varphi_k,
\]
we obtain
\[
\|A\Phi\|_{H^{1/2}}^2
\le \|\Phi\|_{H^{1/2}}^2.
\]
Thus both \((|D_{g_{\mathbb S^n}}|+1)^{-1}\) and \(A\) are bounded on \(H^{1/2}\).

Now let \(\Phi\in W_{\gamma_0}\).  By Lemma~\ref{lem:projection-fixed-gamma}, \(P\Phi=\Phi\).
Since $A$ and $(|D_{g_{\mathbb S^n}}|+1)^{-1}$ are bounded on $H^{1/2}$, we have
\[
  P(A\Phi)=A(P\Phi)=A\Phi
\]
and
\[
  P\bigl((|D_{g_{\mathbb S^n}}|+1)^{-1}\Phi\bigr)
  =(|D_{g_{\mathbb S^n}}|+1)^{-1}P\Phi
  =(|D_{g_{\mathbb S^n}}|+1)^{-1}\Phi.
\]
Both resulting spinors belong to $H^{1/2}$ and are fixed by $P$.
Therefore
\[
  A(W_{\gamma_0})\subset W_{\gamma_0},
  \qquad
  (|D_{g_{\mathbb S^n}}|+1)^{-1}(W_{\gamma_0})\subset W_{\gamma_0}.
\]
\end{proof}
\subsection{Fredholm property and non-degeneracy} We now use the properties of $P$ established above to study
$(F_0^{\gamma_0})''(\psi_\lambda)$ for
$\psi_\lambda\in\mathcal M_{\gamma_0}$.
We first prove that this operator is Fredholm of index 0
and then show that its kernel coincides with
$T_{\psi_\lambda}\mathcal M_{\gamma_0}$.
\begin{lem}
\label{lem:fredholm-property}
For any \(\psi_\lambda\in \mathcal{M}_{\gamma_0 }\), 
\(
(F_0^{\gamma_0})''(\psi_\lambda):E_{\gamma_0}\longrightarrow E_{\gamma_0}
\)
is a bounded self-adjoint Fredholm operator of index \(0\).
\end{lem}
\begin{proof}
By the definition of the restricted Hessian through the Riesz representation
theorem, \((F_0^{\gamma_0})''(\psi_\lambda)\) is a bounded self-adjoint
operator on \(E_{\gamma_0}\).
It remains to prove the Fredholm property.

Set $\widetilde\psi_\lambda:=I(\psi_\lambda)$.
Since $F_0=\mathcal L\circ I$ and
$I|_{E_{\gamma_0}}:E_{\gamma_0}\to W_{\gamma_0}$
is an isomorphism, it suffices to prove that the Hessian of
$\mathcal L|_{W_{\gamma_0}}$ at $\widetilde\psi_\lambda$
is Fredholm of index zero. For $\Phi,\Psi\in H^{1/2}(\mathbb S^n,\Sigma\mathbb S^n)$, we have
\begin{equation}\label{Hessian}
\begin{aligned}
\mathcal L''(\widetilde\psi_\lambda)[\Phi,\Psi]
={}&
\int_{\mathbb S^n}
(D_{g_{\mathbb S^n}}\Phi,\Psi)
\,d\mathrm{vol}_{g_{\mathbb S^n}}
-\frac n2
\int_{\mathbb S^n}
|\widetilde\psi_\lambda|^{2^*-2}(\Phi,\Psi)
\,d\mathrm{vol}_{g_{\mathbb S^n}}
\\
&-\frac{n}{n-1}
\int_{\mathbb S^n}
|\widetilde\psi_\lambda|^{2^*-4}
(\widetilde\psi_\lambda,\Phi)
(\widetilde\psi_\lambda,\Psi)
\,d\mathrm{vol}_{g_{\mathbb S^n}}.
\end{aligned}
\end{equation}Recall that the inner product on $H^{1/2}(\mathbb S^n,\Sigma\mathbb S^n)$ is
\[
(\Phi,\Psi)_{H^{1/2}}
=
\bigl(
|D_{g_{\mathbb S^n}}|^{1/2}\Phi,
|D_{g_{\mathbb S^n}}|^{1/2}\Psi
\bigr)_{L^2}
+
(\Phi,\Psi)_{L^2}.
\]Thus by equation~\eqref{Hessian}, we obtain
\[
\mathcal L''(\widetilde\psi_\lambda)\Phi
={}A\Phi
-\frac n2 (|D_{g_{\mathbb S^n}}|+1)^{-1}
\bigl(
|\widetilde\psi_\lambda|^{2^*-2}\Phi
\bigr)
-\frac{n}{n-1}(|D_{g_{\mathbb S^n}}|+1)^{-1}
\bigl(
|\widetilde\psi_\lambda|^{2^*-4}
(\widetilde\psi_\lambda,\Phi)\widetilde\psi_\lambda
\bigr),
\]
where $A=(|D_{g_{\mathbb S^n}}|+1)^{-1}D_{g_{\mathbb S^n}}$
was defined in Corollary~\ref{cor:commute-borel}.

We first verify that this operator preserves $W_{\gamma_0}$.
The explicit formula for $\psi_\lambda$ gives
\[
\widetilde\psi_\lambda
=
u_\lambda(\theta)\widetilde\gamma_0
+
v_\lambda(\theta)\partial_\theta\cdot\widetilde\gamma_0
\]
for some real-valued smooth functions $u_\lambda,v_\lambda$
on $(0,\pi)$.
For $\Phi\in\mathcal W_{\gamma_0}(\mathbb S^n)$, write
\[
\Phi
=
g_1(\theta)\widetilde\gamma_0
+
g_2(\theta)\partial_\theta\cdot\widetilde\gamma_0,
\qquad
g_1,g_2\in C_c^\infty((0,\pi),\mathbb R).
\]
The Clifford relations give
\[
|\widetilde\psi_\lambda|^2=u_\lambda^2+v_\lambda^2,
\qquad
(\widetilde\psi_\lambda,\Phi)
=u_\lambda g_1+v_\lambda g_2.
\]
Hence both
\(
|\widetilde\psi_\lambda|^{2^*-2}\Phi\)
 and \(
|\widetilde\psi_\lambda|^{2^*-4}
(\widetilde\psi_\lambda,\Phi)\widetilde\psi_\lambda
\)
belong to $\mathcal W_{\gamma_0}(\mathbb S^n)$.
Since $\widetilde\psi_\lambda$ is smooth and nowhere vanishing
on $\mathbb S^n$, these two multiplication maps are bounded
on $L^2$ and $H^1$, and therefore on $H^{1/2}$ by interpolation.
By density, they also preserve $W_{\gamma_0}$.
Together with Corollary~\ref{cor:commute-borel}, this gives
\[
\mathcal L''(\widetilde\psi_\lambda)(W_{\gamma_0})
\subset W_{\gamma_0}.
\]
Since $W_{\gamma_0}$ carries the induced $H^{1/2}$ inner product,
the Hessian of the restricted functional is therefore
\[
\bigl(\mathcal L|_{W_{\gamma_0}}\bigr)''
(\widetilde\psi_\lambda)
=
\mathcal L''(\widetilde\psi_\lambda)|_{W_{\gamma_0}}.
\]

We next show that $A|_{W_{\gamma_0}}$ is an isomorphism.
Since the eigenvalues of $D_{g_{\mathbb S^n}}$ satisfy
$|\lambda_k|\ge \frac{n}{2}$, the numbers
\(
\frac{1+|\lambda_k|}{\lambda_k}
\)
are uniformly bounded and define a bounded inverse of $A$
on $H^{1/2}$.
The function $t\mapsto\frac{1+|t|}{t}$ on the spectrum of
$D_{g_{\mathbb S^n}}$ can be extended to a bounded real-valued
Borel function on $\mathbb R$.
Corollary~\ref{cor:commute-borel} therefore gives
\[
PA^{-1}=A^{-1}P.
\]
Since $A^{-1}$ is bounded on $H^{1/2}$, Lemma~\ref{lem:projection-fixed-gamma} implies
\(
A^{-1}(W_{\gamma_0})\subset W_{\gamma_0}.
\)
Together with $A(W_{\gamma_0})\subset W_{\gamma_0}$,
this proves that $A|_{W_{\gamma_0}}$ is an isomorphism.

Finally, the last two terms in the formula for
$\mathcal L''(\widetilde\psi_\lambda)$ are compact on $H^{1/2}$.
Indeed, the multiplication maps appearing in these terms
are bounded on $H^{1/2}$, while
\[
(|D_{g_{\mathbb S^n}}|+1)^{-1}:
H^{1/2}(\mathbb S^n,\Sigma\mathbb S^n)
\longrightarrow
H^{3/2}(\mathbb S^n,\Sigma\mathbb S^n)
\]
is bounded and the embedding
$H^{3/2}(\mathbb S^n,\Sigma\mathbb S^n)
\hookrightarrow H^{1/2}(\mathbb S^n,\Sigma\mathbb S^n)$
is compact.
Their restrictions to $W_{\gamma_0}$ are therefore compact.

It follows that the restricted Hessian is the sum of an
isomorphism and a compact operator, and hence is Fredholm
of index zero.
Transferring back through $I|_{E_{\gamma_0}}$ proves the result.
\end{proof}
\begin{prop}
\label{prop:nondeg-fixed-gamma}$\mathcal M_{\gamma_0}$ is a non-degenerate critical manifold
of $F_0^{\gamma_0}$ in the sense that, for every
$\psi_\lambda\in\mathcal M_{\gamma_0}$, we have
\[
\ker\left((F_0^{\gamma_0})''(\psi_\lambda)\right)
=
T_{\psi_\lambda}\mathcal M_{\gamma_0}.
\]
\end{prop}\begin{proof}
  Since every $\psi_\lambda\in\mathcal M_{\gamma_0}$ is a critical
point of $F_0^{\gamma_0}$, differentiating
$(F_0^{\gamma_0})'(\psi_\lambda)=0$ with respect to $\lambda$ gives
\(
T_{\psi_\lambda}\mathcal M_{\gamma_0}
\subset
\ker\left((F_0^{\gamma_0})''(\psi_\lambda)\right).
\) Thus to complete the proof of the above proposition, it
suffices to prove the opposite inclusion.
Recall that
\(
\psi_\lambda(x)
=
\lambda^{-\frac{n-1}{2}}\psi_1(\frac{x}{\lambda}).
\)
These dilations preserve $E_{\gamma_0}$ and leave $F_0$
invariant, and hence identify the Hessian kernels at
$\psi_1$ and $\psi_\lambda$. It therefore suffices to prove that
\begin{equation}\label{dimension}
\dim\ker\left((F_0^{\gamma_0})''(\psi_1)\right)\leq 1.
\end{equation}
Let
\[
\eta\in\ker (F_0^{\gamma_0})''(\psi_1),
\qquad
\widetilde\eta:=I(\eta)\in W_{\gamma_0},
\qquad
\widetilde\psi_1:=I(\psi_1).
\]
The explicit formula for $\psi_1$ and the conformal covariance
of the Dirac operator give
\[
|\widetilde\psi_1|=1,
\qquad
D_{g_{\mathbb S^n}}\widetilde\psi_1
=
\frac n2\widetilde\psi_1.
\]
In particular, $\widetilde\psi_1$ is smooth. The equality case in Friedrich's eigenvalue
estimate (see \cite{F2000,G2009}) gives
\begin{equation}\label{killing spinor}
\nabla_X\widetilde\psi_1
=
-\frac12 X\cdot\widetilde\psi_1,
\qquad X\in T\mathbb S^n. 
\end{equation}

As shown in the proof of Lemma~\ref{lem:fredholm-property},
$\mathcal L''(\widetilde\psi_1)$ preserves $W_{\gamma_0}$,
and its restriction to $W_{\gamma_0}$ is the Hessian of
$\mathcal L|_{W_{\gamma_0}}$. Hence
\(
\widetilde\eta
\in
\ker\bigl(\mathcal L''(\widetilde\psi_1)\bigr).
\) Using the formula for $\mathcal L''(\widetilde\psi_1)$
and $|\widetilde\psi_1|=1$, we obtain 
\begin{equation}\label{linearized}
D_{g_{\mathbb S^n}}\widetilde\eta
=
\frac n2\widetilde\eta
+
\frac n{n-1}
(\widetilde\psi_1,\widetilde\eta)\widetilde\psi_1.
 \end{equation}By the formula derived in the proof of Lemma~\ref{lem:Wtheta-characterization} and the
explicit expression for $\psi_1$, we obtain, for $0<\theta<\pi$,
\[
\widetilde\psi_1
=
\sin\frac{\theta}{2}\,\widetilde\gamma_0
+
\cos\frac{\theta}{2}\,
\partial_\theta\cdot\widetilde\gamma_0.
\]
and therefore
\[
\partial_\theta\cdot\widetilde\psi_1
=
-\cos\frac\theta2\,\widetilde\gamma_0
+
\sin\frac\theta2\,\partial_\theta\cdot\widetilde\gamma_0.
\]
Thus, by Lemmas~\ref{lem:Wtheta-characterization} and~\ref{lem:projection-fixed-gamma}, we may write
\[
\widetilde\eta
=
a(\theta)\widetilde\psi_1
+
b(\theta)\partial_\theta\cdot\widetilde\psi_1
\]
for real-valued functions $a,b\in C^\infty((0,\pi))$. Since $\widetilde\psi_1$ and
$\partial_\theta\cdot\widetilde\psi_1$ are pointwise
orthonormal with respect to the real inner product, we have
\[
a(\theta)=(\widetilde\eta,\widetilde\psi_1).
\]
Thus $a$ extends to a smooth real-valued function on
$\mathbb S^n$.
Choose a local orthonormal frame $\{e_i\}_{i=1}^n$ with
$e_1=\partial_\theta$. On $0<\theta<\pi$, we have
\[
\nabla_{\partial_\theta}\partial_\theta=0,
\qquad
\nabla_{e_j}\partial_\theta=\cot\theta\,e_j,
\quad j=2,\ldots,n.
\]
Using \eqref{killing spinor}, we compute
\[
\begin{aligned}
D_{g_{\mathbb S^n}}\widetilde\eta
&=
\sum_{i=1}^n e_i\cdot\nabla_{e_i}
\bigl(
a\widetilde\psi_1
+b\partial_\theta\cdot\widetilde\psi_1
\bigr)\\
&=
a'\partial_\theta\cdot\widetilde\psi_1
+\frac n2 a\widetilde\psi_1
-b'\widetilde\psi_1
+b\cot\theta\sum_{j=2}^n
e_j\cdot e_j\cdot\widetilde\psi_1
-\frac b2\sum_{i=1}^n
e_i\cdot\partial_\theta\cdot e_i
\cdot\widetilde\psi_1\\
&=
\left(
\frac n2 a-b'-(n-1)\cot\theta\,b
\right)\widetilde\psi_1
+
\left(
a'-\frac{n-2}{2}b
\right)\partial_\theta\cdot\widetilde\psi_1.
\end{aligned}
\]
Substituting this expression and
$(\widetilde\psi_1,\widetilde\eta)=a$
into equation~\eqref{linearized} and comparing the two
components, we obtain
\[
\begin{cases}
-b'-(n-1)\cot\theta\,b=\dfrac{n}{n-1}a,\\[4pt]
a'=(n-1)b.
\end{cases}
\]
Eliminating $b$, we obtain
\[
a''+(n-1)\cot\theta\,a'+na=0.
\]
Since $a$ is smooth on $\mathbb S^n$, this means
\[
-\Delta_{g_{\mathbb S^n}}a=na.
\]
The eigenspace of $-\Delta_{g_{\mathbb S^n}}$ corresponding
to the eigenvalue $n$ is spanned by the coordinate functions
$y^1,\ldots,y^{n+1}$ restricted to $\mathbb S^n$.
Since
\(
y=(\sin\theta\,\omega,\cos\theta)
\)
and $a$ depends only on $\theta$, it follows that
\[
a(\theta)=c\cos\theta
\]
for some $c\in\mathbb R$. Thus
\[
b(\theta)=-\frac c{n-1}\sin\theta.
\]
Consequently, every kernel element satisfies
\[
\widetilde\eta
=
c\left(
\cos\theta\,\widetilde\psi_1
-
\frac{\sin\theta}{n-1}
\partial_\theta\cdot\widetilde\psi_1
\right).
\]By the injectivity of $I$, \eqref{dimension} follows.
Since $T_{\psi_1}\mathcal M_{\gamma_0}$ is one-dimensional,
the inclusion proved at the beginning yields
\[
\ker\left((F_0^{\gamma_0})''(\psi_1)\right)
=
T_{\psi_1}\mathcal M_{\gamma_0}.
\]
The conclusion for every $\lambda>0$ follows by dilation.
\end{proof}
\section{Proofs of the main theorems}\label{sec:The reduced problem and existence results}
As shown in Section~\ref{sec:The fixed-spinor radial setting}, $F_0^{\gamma_0}$ and
$\mathcal M_{\gamma_0}$ satisfy assumptions (A1)--(A3)
of Proposition~\ref{reduction}. Thus, by Proposition~\ref{reduction} and Lemma~\ref{lem:natural-constraint-fixed-gamma},
to obtain solutions of equation~\eqref{(1.1)} for sufficiently small
$\varepsilon>0$, it suffices to find a stable critical
point in $(0,\infty)$ of
\[
\begin{aligned}
\Gamma(\lambda)
&=-\frac{4}{n-1}G(\psi_\lambda)
=\int_{\mathbb R^n}
K(|x|)|\psi_\lambda(x)|^{2^*}\,dx\\
&=2^n\int_{\mathbb R^n}
K(|x|)\frac{\lambda^n}{(\lambda^2+|x|^2)^n}\,dx.
\end{aligned}
\]
Set
\[
z_0(r):=\frac{2^{\frac{n-1}{2}}}
{(1+r^2)^{\frac{n-1}{2}}},
\qquad r\geq0.
\]
Using polar coordinates and then a change of variables,
we obtain
\begin{equation}\label{rf}
\Gamma(\lambda)
=\omega_n\lambda^{-n}
\int_0^\infty K(r)
z_0\left(\frac r\lambda\right)^{2^*}r^{n-1}\,dr=\omega_n
\int_0^\infty K(\lambda r)z_0(r)^{2^*}r^{n-1}\,dr,
\end{equation}where $\omega_n=|\mathbb S^{n-1}|$. By \eqref{ass:K1}, we have:
\begin{enumerate}
\item[(a)]
$\Gamma$ extends continuously to $\lambda=0$ by setting
\(
\Gamma(0)
=
\omega_n K(0)\int_0^\infty z_0(r)^{2^*}r^{n-1}\,dr.
\)

\item[(b)]
$\Gamma'(0)=0$, and $\Gamma$ admits a $C^1$ extension
to $\mathbb R$ by symmetry.
\end{enumerate}
The following lemmas establish further properties of $\Gamma$.
\begin{lem}\label{lem 4.1}
Suppose that $K$ satisfies \eqref{ass:K1}. Then $\Gamma(\lambda)\rightarrow 0$ as $\lambda\rightarrow +\infty$.
\end{lem}
\begin{proof}
From \eqref{rf}, we have
$$\begin{aligned}
	|\Gamma({\lambda})|&=\omega_n2^n\left|\int_{0}^{\infty}K(r)\frac{\lambda^nr^{n-1}}{(\lambda^2+r^2)^n}d r\right|\\
	&\leq \omega_n2^n\lambda^{-n}\int_{0}^{1}|K(r)|r^{n-1}d r +\omega_n2^n\lambda^{\alpha-n}\int_{1}^{\infty}\frac{|K(r)|}{r^\alpha}r^{n-1}d r.
\end{aligned}$$	
Since $ r^{-\alpha} K(r) \in L^1\left([1, \infty), r^{n-1} d r\right)$ and $\alpha<n$, the result follows.
\end{proof}
\begin{lem}\label{neq}
Under the preceding assumptions, if \(K\not\equiv0\), then
\(
\Gamma\not\equiv0.
\)
\end{lem}

\begin{proof}For $g \in L^1([0, \infty), d r / r)$, we define the Mellin transform by
$$
\mathscr{M}[g](s)=\int_0^{\infty} r^{-i s} g(r) \frac{d r}{r},
$$
see \cite{F1992}. The associated convolution is defined by
$$
(g \times h)(t)=\int_0^{\infty} g(r) h\left(\frac{t}{r}\right) \frac{d r}{r}.
$$
Then $\mathscr{M}[g \times h]=\mathscr{M}[g] \cdot \mathscr{M}[h]$. With this notation, $\Gamma$ can be written as
$$
\Gamma(\lambda)=\omega_n\int_0^{\infty} K(r) z_0^{2^*}\left(\frac{r}{\lambda}\right)\left(\frac{r}{\lambda}\right)^n \frac{d r}{r} .
$$Next, set $m=n-\alpha$ and$$
g(r)=K(r) r^m, \quad h(r)=z_0^{2^*}\left(\frac{1}{r}\right)\left(\frac{1}{r}\right)^{n-m} .
$$
Observe that $g, h \in L^1([0, \infty), d r / r)$. It follows that $\Gamma(\lambda)=\omega_n\lambda^{-m}(g \times h)(\lambda)$. Suppose, by contradiction, that \(\Gamma\equiv0\). Then $g \times h \equiv 0$ and hence
$$
\mathscr{M}[g] \cdot \mathscr{M}[h]=\mathscr{M}[g \times h] \equiv 0 .
$$
On the other hand, 
\[
\mathscr{M} [h](s)
=
2^{n-1}
B\left(
\frac{n+m-is}{2},
\frac{n-m+is}{2}
\right),
\]
which is nonzero for every \(s\in\mathbb R\). Therefore, $\mathscr{M}[g] \equiv 0$. It follows that \(g\equiv0\), and hence \(K\equiv0\), a contradiction.
\end{proof}
We next establish an asymptotic formula for $\Gamma'$,
which will be used to prove Theorem~\ref{thm:intro-existence} in case (iii).

\begin{lem}\label{lem 4.3}
Suppose that
\(
K'(r)r\in L^1((0,\infty),r^{n-1}dr),
\)
and set
\(
\kappa:=\int_0^\infty K'(r)r^n\,dr.
\)
Then, if \(\kappa<0\), respectively \(\kappa>0\), one has
\[
\Gamma'(\lambda)\to0^-,
\qquad\text{respectively }0^+,
\quad\text{as }\lambda\to+\infty.
\]
\end{lem}

\begin{proof}
Applying the change of variables \(s=\lambda r\) to \eqref{rf}, we obtain
\[
\Gamma'(\lambda)
=
\omega_n\lambda^{-n-1}
\int_0^\infty
K'(s)
z_0^{2^*}\left(\frac{s}{\lambda}\right)
s^n\,ds.
\]
Since
\[
z_0^{2^*}\left(\frac{s}{\lambda}\right)\to z_0^{2^*}(0)=2^n
\]
and
\[
\left|
K'(s)
z_0^{2^*}\left(\frac{s}{\lambda}\right)
s^n
\right|
\le
C|K'(s)|s^n,
\]
with
$
|K'(s)|s^n\in L^1(0,\infty),
$
the dominated convergence theorem yields
\[
\lambda^{n+1}\Gamma'(\lambda)
\to
\omega_n2^n
\int_0^\infty K'(s)s^n\,ds
=
\omega_n2^n\kappa.
\]
Therefore
\[
\Gamma'(\lambda)
=
\omega_n2^n\kappa\,\lambda^{-n-1}
+
o(\lambda^{-n-1}).
\]
The conclusion follows immediately.
\end{proof}
\begin{proof}[\textbf{Proof of Theorem~\ref{thm:intro-existence}}]
We consider conditions (i)--(iii) separately.

\smallskip
\noindent\textit{Case (i): $K(0)=0$ and $K\not\equiv0$.} By Lemmas~\ref{lem 4.1} and~\ref{neq},
\[
\Gamma(0)=0,
\qquad
\lim_{\lambda\to+\infty}\Gamma(\lambda)=0,
\qquad
\Gamma\not\equiv0.
\]
Hence $\Gamma$ attains a positive maximum or a negative
minimum at some $\bar\lambda>0$.

\smallskip
\noindent\textit{Case (ii): $K\in C^2(\mathbb R^n)$ and
$K(0)K''(0)>0$.}
We distinguish between $n\ge3$ and $n=2$.

Suppose first that $n\ge3$.
Since $K$ is radial and of class $C^2$, we have $K'(0)=0$.
Together with the boundedness of $K$, this gives
\[
|K(r)-K(0)|\le Cr^2,
\qquad r\ge0.
\]
Moreover,
\[
\int_0^\infty z_0(r)^{2^*}r^{n+1}\,dr<\infty.
\]
Taylor's formula and dominated convergence applied to
\eqref{rf} therefore yield
\[
\Gamma(\lambda)
=
\Gamma(0)
+
\frac{\omega_n}{2}K''(0)\lambda^2
\int_0^\infty z_0(r)^{2^*}r^{n+1}\,dr
+
o(\lambda^2)
\qquad\text{as }\lambda\to0^+.
\]

Now suppose that $n=2$. In this case,
\[
\Gamma(\lambda)
=
4\omega_2\lambda^2
\int_0^\infty K(r)\frac{r}{(\lambda^2+r^2)^2}\,dr,
\qquad
\Gamma(0)=2\omega_2K(0).
\]
Consequently,
\[
\Gamma(\lambda)-\Gamma(0)
=
4\omega_2\lambda^2
\int_0^\infty
(K(r)-K(0))\frac{r}{(\lambda^2+r^2)^2}\,dr.
\]
Write
\[
K(r)-K(0)
=
\frac12K''(0)r^2+r^2\rho(r),
\qquad
\rho(r)\to0
\quad\text{as }r\to0^+.
\]
Given $\eta>0$, choose $\delta>0$ such that
$|\rho(r)|\le\eta$ for $0<r\le\delta$.
A direct computation gives
\[
\begin{aligned}
\int_0^\delta\frac{r^3}{(\lambda^2+r^2)^2}\,dr
=
\frac12\left[
\log\frac{\lambda^2+\delta^2}{\lambda^2}
+
\frac{\lambda^2}{\lambda^2+\delta^2}-1
\right]=
\log\frac1\lambda+O(1)
\qquad\text{as }\lambda\to0^+.
\end{aligned}
\]
Thus,
\[
\left|
\lambda^2\int_0^\delta
\rho(r)\frac{r^3}{(\lambda^2+r^2)^2}\,dr
\right|
\le
\eta\lambda^2
\left(\log\frac1\lambda+O(1)\right).
\]
On the other hand, the boundedness of $K$ implies
\[
\left|
\lambda^2\int_\delta^\infty
(K(r)-K(0))\frac{r}{(\lambda^2+r^2)^2}\,dr
\right|
\le
C\lambda^2\int_\delta^\infty r^{-3}\,dr
=
O(\lambda^2).
\]
Since $\eta>0$ is arbitrary, these estimates yield
\[
\Gamma(\lambda)
=
\Gamma(0)
+
2\omega_2K''(0)\lambda^2\log\frac1\lambda
+
o\left(\lambda^2\log\frac1\lambda\right)
\qquad\text{as }\lambda\to0^+.
\]

In both cases, if $K(0)>0$, then $K''(0)>0$ and
\(
\Gamma(\lambda)>\Gamma(0)>0
\)
for all sufficiently small $\lambda>0$.
Together with Lemma~\ref{lem 4.1}, this shows that $\Gamma$ attains
a positive global maximum in $(0,\infty)$.
If $K(0)<0$, the same argument shows that $\Gamma$
attains a negative global minimum in $(0,\infty)$.

\smallskip
\noindent\textit{Case (iii):}
Assume that
\(
\int_0^\infty |K'(r)|r^n\,dr<\infty\),
\(
K(0)\kappa>0,
\)
where $\kappa$ is defined in Lemma~\ref{lem 4.3}.

If $K(0)>0$, then $\kappa>0$, and Lemma~\ref{lem 4.3} gives
$\Gamma'(\lambda)>0$ for all sufficiently large $\lambda$.
Since $\Gamma(\lambda)\to0$ as $\lambda\to+\infty$,
we have $\Gamma(\lambda)<0$ for all sufficiently large
$\lambda$.
As $\Gamma(0)>0$, it follows that $\Gamma$ attains
a negative global minimum in $(0,\infty)$.
If $K(0)<0$, then $\kappa<0$, and the same argument
gives a positive global maximum in $(0,\infty)$.

Since $\Gamma$ is nonconstant and real analytic on
$(0,\infty)$, the extremum obtained in each case is
a stable critical point.
The conclusion follows from Proposition~\ref{reduction} and Lemma~\ref{lem:natural-constraint-fixed-gamma}.
\end{proof}
\begin{proof}[\textbf{Proof of Theorem~\ref{thm:intro-multiplicity}}]
Suppose first that \emph{(i)} holds. If
\(\Gamma(1)>0\), then \(K(0)\leq0\), and hence
\(
\Gamma(0)\leq0.
\)
Also, by Lemma~\ref{lem 4.1},
$
\lim_{\lambda\to\infty}\Gamma(\lambda)=0.
$
Since \(\Gamma(1)>0\), the function \(\Gamma\) attains a positive
maximum at some point in \((0,\infty)\). If
\(\Gamma(1)<0\), the same argument, with maxima replaced by minima,
shows that \(\Gamma\) attains a negative minimum in
\((0,\infty)\).

Suppose next that \emph{(ii)} holds. If
\(\Gamma'(1)>0\), then
\(
\Gamma(\lambda)>0
\)
for some \(\lambda>1\) sufficiently close to \(1\). Since
\(\Gamma(\lambda)\to0\) as \(\lambda\to\infty\), the function
\(\Gamma\) attains a positive maximum in \((1,\infty)\). The case
\(\Gamma'(1)<0\) is analogous and yields a negative minimum in
\((1,\infty)\).
Moreover, assume that
\(\Gamma'(1)K(0)\geq0\). If \(\Gamma'(1)>0\), then
\(K(0)\geq0\), while
\(
\Gamma(\lambda)<0
\)
for some \(\lambda<1\) sufficiently close to \(1\). Hence
\(\Gamma\) attains a negative minimum in \((0,1)\). The case
\(\Gamma'(1)<0\) similarly yields a positive maximum in
\((0,1)\). As in the proof of Theorem~1\ref{thm:intro-existence}, the extrema obtained above
are stable critical points.
The existence conclusion follows from Proposition~\ref{reduction}
and Lemma~\ref{lem:natural-constraint-fixed-gamma}.
In the multiplicity case, the two critical points lie
in the disjoint intervals $(0,1)$ and $(1,\infty)$.
Applying the reduction in disjoint neighborhoods of
the corresponding bubbles yields two distinct solutions.
\end{proof}
\begin{rem}
Our results hold for every choice of the unit spinor
$\gamma_0$. In particular, the two solutions in
Theorem~\ref{thm:intro-multiplicity} are obtained for the same fixed $\gamma_0$.
Radial symmetry allows us to reduce the existence problem
to finding critical points of the function $\Gamma(\lambda)$.
A natural further question is whether the solutions
constructed here persist under small nonradial
perturbations of $K$.
\end{rem}

\end{document}